\documentclass{siamltex}

\usepackage[%
  colorlinks%
  ,bookmarks={false}%
  ,pdftitle={}%
  ,pdfsubject={}%
  ,pdfkeywords={}%
  ,pdfauthor={Tom Maerz <tom.maerz@gmx.de>}%
  ]{hyperref}

\usepackage{amsmath}
\usepackage{amssymb}
\usepackage{mathpazo}
\usepackage{graphicx}
\usepackage{enumerate}
\usepackage[latin1]{inputenc}
\usepackage[font=footnotesize]{subcaption}
\usepackage[labelfont=bf,font=small,labelsep=space]{caption}
\usepackage[left=1.5in,right=1.5in, top=1.5in]{geometry}
\usepackage[ruled,commentsnumbered]{algorithm2e}
\usepackage{txfonts}

\newcommand{\R}{{\mathbb  R}}

\newcommand{\laplace}{\Delta}

\newcommand{\BV}{{BV}}
\newcommand{\Lan}{{\mathcal L}}
\newcommand{\sat}{{\text{sat}}}
\newcommand{\vf}[1]{{\mathbf{#1}}} 
\newcommand{\bnd}[1]{{\vf{\hat{#1}}}} 
\newcommand{\nd}[1]{{\hat{#1}}} 
\newcommand{\rd}[1]{{\bar{#1}}} 
\newcommand{\brd}[1]{{\vf{\bar{#1}}}}
\newcommand{\MPI}{{\vf{MPI}}}

\DeclareMathOperator{\diver}{div}
\DeclareMathOperator{\trace}{trace}

\DeclareMathOperator{\sign}{sign}

\title{Model-Based Reconstruction for Magnetic Particle Imaging in 2D and 3D}
\author{
Thomas M\"{a}rz\thanks{
Mathematical Institute,
University of Oxford, UK.
Email: {\tt maerz@maths.ox.ac.uk}.} 
\and Andreas Weinmann\thanks{
Department of Mathematics, 
Technische Universit\"at M\"unchen,
and Helmholtz Center Munich, Germany.
Email: {\tt andreas.weinmann@tum.de}.} 
}

\begin{document}
\maketitle

\begin{abstract}
We contribute to the mathematical modeling and analysis  
of magnetic particle imaging which is a promising new in-vivo imaging modality.
Concerning modeling, we develop a structured decomposition of the imaging process
and extract its core part which we reveal to be common to all previous contributions
in this context. The central contribution of this paper is the development of reconstruction formulae for MPI in 2D and 3D.  
Until now, in the multivariate setup, only time consuming measurement approaches are available,
whereas reconstruction formulae are only available in 1D.
The 2D and the 3D (describing the real world) reconstruction formulae which we derive here are significantly different from the 1D situation --   
in particular there is no Dirac property in dimensions greater than one when the particle sizes approach zero.
As a further result of our analysis, we conclude that the reconstruction problem in MPI is severely ill-posed. 
Finally, we obtain a model-based reconstruction algorithm.
\end{abstract}



\section{Introduction}

Magnetic particle imaging (MPI) is an emerging imaging mo\-dality which was developed by \mbox{Gleich} and Weizenecker in 2005 \cite{GleichWeizenecker2005}. 
Its goal is to determine the spatial distribution of magnetic nanoparticles by measuring the non-linear magnetization response of the particles to an applied magnetic field. 
For example, in angiographic applications, the particles are injected in humans' systems of blood vessels in order to image this system 
by determining the concentration of the particles.

The particles are not observed directly, but a certain current induced by them is measured. 
More precisely, the non-linear response of superparamagnetic nanoparticles to a temporally changing applied magnetic field is exploited: 
Changing the magnetic field results in a change of magnetization locally near a point, say $x$; 
the local change of magnetization induces a current in a set of receive coils where the current reflects the concentration of the particles near $x.$  

MPI is a very promising imaging modality for biomedical diagnostics
because of two features:
firstly, MPI offers a high dynamic spatial and temporal resolution
\cite{KnoppBiederer_etal2010}.
Secondly, in contrast to many other tomographic methods such as PET or SPECT
(cf. \cite{SPECT,PET}), MPI does not employ any ionizing radiation.   
In particular, patients are not exposed to radioactive substances in the imaging process. 
Furthermore, only the magnetic particles are imaged which avoids artefacts stemming from nearby tissue as in angiographic CT. 
Potential applications of MPI are any tracer based diagnostics, such as blood flow imaging and cancer detection \cite[Ch.\! 7]{BuzugKnopp2012} as well as quantitative stem cell imaging \cite{zheng2013quantitative}. 
A basic reference on MPI is the just mentioned book \cite{BuzugKnopp2012}, as well as the thesis 
\cite{KnoppDiss2010}. Concerning early work, we exemplarily refer to 
\cite{GleichWeizeneckerBorgert2008, Weizenecker_etal2009, Rahemeretal2009}.
As further references we would also like to mention \cite{knopp2011prediction, Knopp_etal2011hm,Knopp_etal2010ec, sattel2009single, goodwill2012x,saritas2013magnetic, konkle2011development,ferguson2009optimization}. These references also include references to related imaging setups and to related questions concerning particle design as well.

In a proper imaging setup, the MPI forward operator is well modeled as a linear operator which has to be inverted (in the sense of inverse problems) to reconstruct the signal from the measured voltage.
The present approaches to represent this forward operator can be categorized into measurement-based approaches,
 e.g., \cite{WeizeneckerBorgertGleich2007, Rahmer_etal2012, Lampe_etal2012}
and model-based approaches,
 e.g., \cite{KnoppSattel_etal2010, Gruettner_etal2013}. 
In the measurement based approach, a basis (or, more general, a dictionary) is considered and, for each member, the system response is measured physically.
This is rather expensive; in particular, when reconstructing with higher resolution. 
This is actually a time limiting factor in practice.
Knowing these measurements, the reconstruction consists of the (regularized) solution of the corresponding system of linear equations with the rows of the corresponding matrix being the measured system responses. In order to reduce the time consumption, a model-based approach is desirable. 

Presently, there are basically three approaches to derive a model for the MPI forward operator which are intimately related.
They have been developed over the past five years; for a detailed overview, we refer to \cite{Gruettner_etal2013}.
All these models employ Faraday's law of induction with respect to a volumetric coil; 
see \cite{BuzugKnopp2012,KnoppDiss2010} for a derivation. 
Combined with the Langevin model of the magnetization response a description of the received signal is obtained. 
The applied magnetic field is in general a combination of a static field and a time-dependent drive field. 
The latter determines the movement of the so-called field free point that scans the test object.

Let us be more precise.
In \cite{Rahemeretal2009}, Rahmer et al.\ consider a 1D scenario and study the Fourier series of the signal derived from a delta distribution. 
They conclude that, in the case of ideal particles, the system function can be written in terms of Chebyshev polynomials.
In the 2D/3D scenario such a representation is 
not available \cite{Rahmer_etal2012}.
In contrast to \cite{Rahemeretal2009}, Goodwill and Conolly  \cite{GoodwillConolly2010, GoodwillConolly2011} 
and Schomberg \cite{Schomberg2010} follow a geometric approach. 
In \cite{GoodwillConolly2010, GoodwillConolly2011}, the received signal is studied in terms of the trajectory of the field free point
whereas the approach of \cite{Schomberg2010} studies the integral of the received signal instead of the signal itself.
The first authors observe that the field free point 
in the spatial domain (or x-space in their terminology) 
is uniquely determined 
by the magnetic drive field and vice versa. 

Their main result is that the signal generation process can be described by a certain kind of convolution operator
based on a matrix-valued point-spread function or kernel (cf. the ``Generalized MPI signal equation'' in \cite{GoodwillConolly2011}) involving the field-free point and the tangent to its trajectory. 

Schomberg \cite{Schomberg2010} obtains a formulation of convolution type in which the kernel describes the magnetization response of the system. In contrast to 
\cite{GoodwillConolly2010, GoodwillConolly2011}, the kernel is vector-valued and
a different quantity is described.  
For the 1D case, it is shown by Gr\"uttner et al. \cite{Gruettner_etal2013} that the x-space formulation is mathematically equivalent to the frequency space formulation of the model in \cite{Rahemeretal2009}. 

Albeit some previous approaches allow for arbitrary trajectories, to our knowledge, no previous approach is detached from them. Furthermore, in contrast to 1D, there are no reconstruction formulae and, thus, no algorithmic approaches based on such formulae are available. The previous model-based reconstruction algorithms in 2D and 3D mainly consist of calculating the entries of the corresponding matrices instead of measuring them. 

Certainly, reconstruction formulae in 2D and 3D are very desireable. We believe that such formulae have not been derived yet for the following reason: in contrast to 1D, the problem cannot be modeled  as a standard time-invariant system which would result in a standard deconvolution problem.
Instead, a matrix-valued kernel enters the scene which modifies the problem in such a way, that, in the recent years,  no solution was figured out -- although desired. As it turns out in this paper, 
a simple but striking observation allows us to develop an approach which solves the problem.
Indeed, applying this striking observation casts the problem into a simpler problem one can deal with easier.

\paragraph*{Contributions} 
In this paper we make three contributions:
(i) a unified spatial model in terms of an MPI core operator;
(ii) reconstruction formulae for multivariate MPI in 2D and 3D;
(iii) reconstruction algorithms for multivariate MPI in 2D and 3D.

Concerning (i), we revisit the modeling of the MPI time signal generation process and decompose it.
We identify an MPI core operator \eqref{eq:MPIcoreOperator} acting on phase space.
The phase space point of view has the advantage of being totally free of the concept of any trajectories. Trajectories only enter the scene as a means of getting data for 
the operator equation involving the core operator.
This point of view is essential for understanding the spatial action of an MPI scanner on a spatial
particle distribution (the image of interest) and is key for the derivation of our reconstruction formulae.
The MPI core operator is formulated in a common form for all three cases 1D, 2D and 3D.

The most important contribution of this paper is 
the derivation of reconstruction formulae 
(Theorems~\ref{thm:RecoIdeal} and \ref{thm:RecoNonIdeal})
for multivariate MPI in 2D and 3D (ii). 
This is a crucial step to bring a model-based reconstruction approach for dimensions higher than one to the MPI practinoners.  
Until now, such formulae only existed in 1D (cf. \cite{Rahemeretal2009,Rahmer_etal2012}), while for 2D and 3D only extremely time consuming measurement approaches have been employed in these practically relevant situations.
These formulas are a result of our analysis of the MPI core operator.
Our analysis shows that
the 1D situation is significantly different from the 2D and 3D situation;
we also emphasize that the 2D and the 3D situation themselves are also significantly different. 
This difference is owed to a close relationship between the MPI core operator and the fundamental solution of Laplace's equation
which we highlight in Theorem~\ref{thm:IdealizationTheorem}.
Our analysis also reveals the severe ill-posedness of the MPI reconstruction problem (Theorem \ref{thm:SevIllposed}) which explains why regularization is needed. 

Taking this into account, we build reconstruction algorithms (Algorithm~\ref{alg:recAlgo}) for multivariate MPI which are based on our reconstruction formulae (iii).
In this way, we provide a model-based reconstruction approach for the 2D and 3D case.
To our knowledge, this is the first approach based on a reconstruction formula for the 
multivariate case.

\paragraph*{Outline of the paper}
In Section \ref{sec:MathematicalModel}, 
we provide a thorough mathematical description of the MPI signal encoding. 
In particular, we provide a trajectory free description revealing the core of the MPI signal generation process.
In Section~\ref{sec:AnaysisMPIop}, we analyze this MPI core operator. Here, we 
consider first a high resolution idealization limit in Section~\ref{sec:idealCoreOp}.
Using the derived results, we consider the non-ideal case in Section~\ref{sec:nonIdealCoreOp}. 
In Section~\ref{sec:ReconstructionAlgo}, we apply the derived reconstruction formulas 
for the MPI core operator to establish a reconstruction algorithm. 
In Section~\ref{subsec:disc}, we start with the discretization;
then, we propose a reconstruction algorithm in Section~\ref{subsec:recAlgoDesc};
finally, we illustrate our algorithm with a numerical example in Section~\ref{subsec:numExperiment}.

\section{Mathematical Model of the MPI Signal Encoding}
\label{sec:MathematicalModel}

\begin{figure}[t]
\begin{center} 
\includegraphics[width=0.5\linewidth]{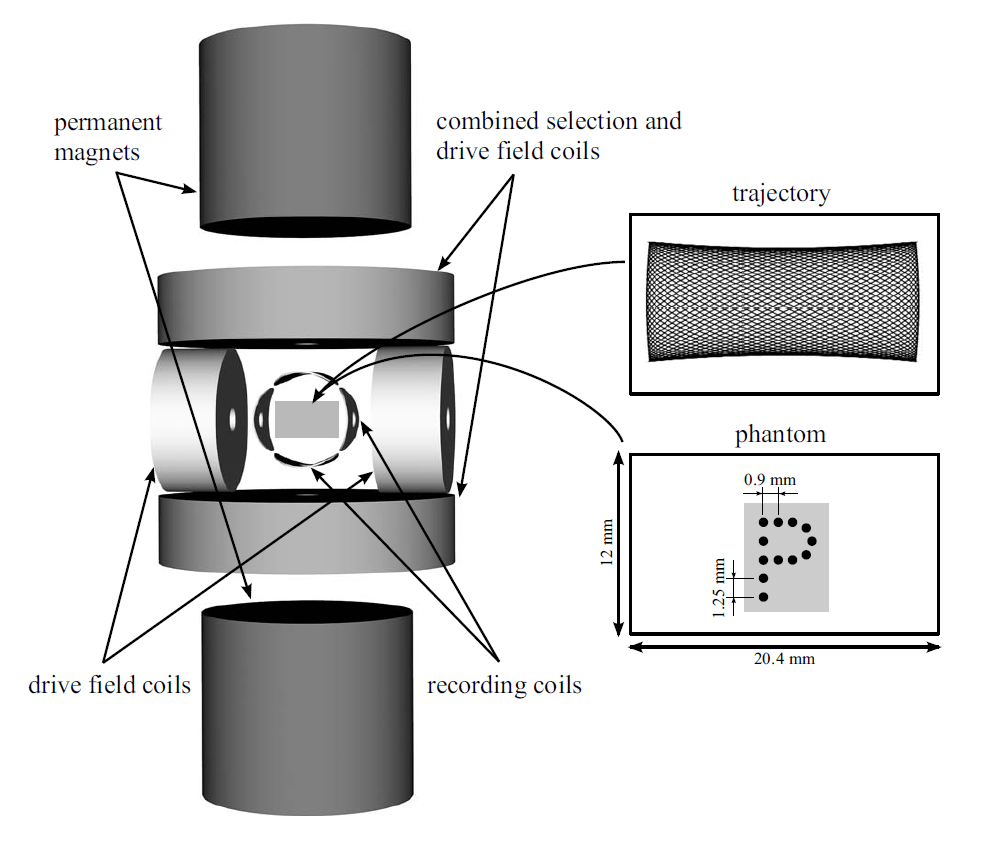}
\caption{Schematic of a 2D MPI scanner (courtesy of \cite{KnoppBiederer_etal2010}).
The rectangle in the center of the scanner is the field of view (FOV).
Two permanent magnets and the outer coil pair aligned with the $y$-direction generate the selection field $\vf{H}^S$.
Both outer coil pairs are used to generate the drive field $\vf{H}^D$ which moves the field free point (FFP) in the FOV for example on Lissajous trajectories.
The inner coil pairs record the signal.}
\label{fig:MPIScannerSchematic}
\end{center}
\end{figure}

We start by providing the basic ideas and principles of MPI in Section \ref{subsec:Basics}. 
We discuss briefly the one dimensional situation of particle detection in Section \ref{forwMPI1d} 
and then turn to the multivariate situation which is the central topic of this paper.
In Section \ref{subsec:MPIForwHighdim} we derive the model describing the MPI signal generation.
Then we reveal the core of the MPI signal generation process in Section \ref{subsec:DecomposingMPIo}.
In Section \ref{subsec:RelExistingModels}  we relate to the existing models.

\subsection{Basic Ideas and Principles}
\label{subsec:Basics}

The idea of magnetic particle imaging is to leverage the non-linear response of superparamagnetic nanoparticles to a temporally changing applied
magnetic field $\vf{H}(x,t)$. Figure~\ref{fig:MPIScannerSchematic} shows a schematic of a 2D MPI scanner. The combination of the permanent magnets and the outer two coil pairs
generates the magnetic field $\vf{H}(x,t)$ while the inner coil pairs record the signal. 
(A similar set-up for a 3D MPI scanner is possible by using three outer
and three inner coil pairs.)

The recording coils receive a voltage $\vf{u}(t)$ which is described by Farraday's law of induction as the temporal
change of the magnetic flux $\vf{\Phi}(t)$, i.e.
\begin{align}\label{eqn:Faraday}
	\vf{u}(t) &= - \frac{d}{dt} \vf{\Phi}(t), & \vf{\Phi}(t) &= \mu_0 \int\limits_{\R^3} \vf{R}(x) (\vf{H}(x,t) + \vf{M}(x,t)) \; dx.
\end{align}
The flux $\vf{\Phi}(t)$ of \eqref{eqn:Faraday} is caused by the applied field $\vf{H}(x,t)$ as well as the magnetization response $\vf{M}(x,t)$. 
The parameter $\mu_0$ denotes the magnetic permeability and $\vf{R}(x) \in \R^{3 \times 3}$ the sensitivity pattern of the three recording coil pairs.  

The second important principle is the Langevin theory of paramagnetism \cite{chikazumi1978physics,jiles1998introduction}:
it gives a description of the magnetization response $\vf{M}(x,t)$
for superparamagnetic nanoparticles in terms of the Langevin function $\Lan$:
\begin{align}\label{eqn:magnet}
	\vf{M}(x,t) &= \rho(x) \; m \; \Lan \left( \frac{|\vf{H}(x,t)|}{H_\sat} \right) \; \frac{\vf{H}(x,t)}{|\vf{H}(x,t)|}, & \Lan(x) &= \coth(x) - \frac{1}{x}.
\end{align}
Here $m$ is the magnetic moment of a single particle, $\rho$ is the concentration of the particles. 
Moreover, the magnetization field $\vf{M}$ is aligned with the applied field $\vf{H}$ and vanishes if $\vf{H}$ is zero.
The goal of magnetic particle imaging is to reconstruct the spatial particle concentration $\rho$ from the received time signal $\vf{u}(t)$.
Since only the magnetization response depends on $\rho$ the signal $\vf{u}(t)$ is reduced to the wanted signal $\vf{s}(t)$,
\begin{align}\label{eq:MPIproblemInTime}
	\vf{s}(t) &= - \mu_0 \frac{d}{dt} \int\limits_{\R^3} \vf{R}(x) \vf{M}(x,t) \; dx,
\end{align}
by subtraction of a reference voltage \cite{Schomberg2010} obtained from a scan of an empty field of view (FOV), i.e., with $\vf{M}=0$.

\subsection{The MPI Principle of Particle Detection in 1D}
\label{forwMPI1d}
The effect of the non-linear response of superparamagnetic nanoparticles is explained best in a simplified one-dimensional scenario, cf. Figure \ref{fig:Response}.
Here, we assume a point sample is located at position $x$, i.e., $\rho(y) = \delta(x-y)$. Moreover, we have a single pair of selection and drive coils such that $\vf{H}(x,t) = H(x,t) e_1$
and a single pair of recording coils such that $\vf{R}(x) = R(x) e_1^T$. With this configuration the signal is 
\begin{align}
	s(t) &= - \mu_0 \; m \; R(x) \frac{d}{dt} \Lan \left( \frac{H(x,t)}{H_\sat} \right)
\end{align}
where we used the fact that the Langevin function is a sigmoid (yellow curve in \ref{fig:Response}).  
We assume now that $H(x,t)$ as function of time $t$ oscillates with small amplitude about a reference value $\hat{H}$.
If the reference value $\hat{H}$ is about zero we will receive a non-trivial signal because of the non-linearity of the Langevin function near zero.
Otherwise if $\hat{H}$ is large in magnitude we will receive the zero signal since the Langevin function is almost constant away from zero.
Thus, control over the applied field allows for the detection of the point sample.   
\begin{figure}[t]
\begin{center} 
\includegraphics[width=0.5\linewidth]{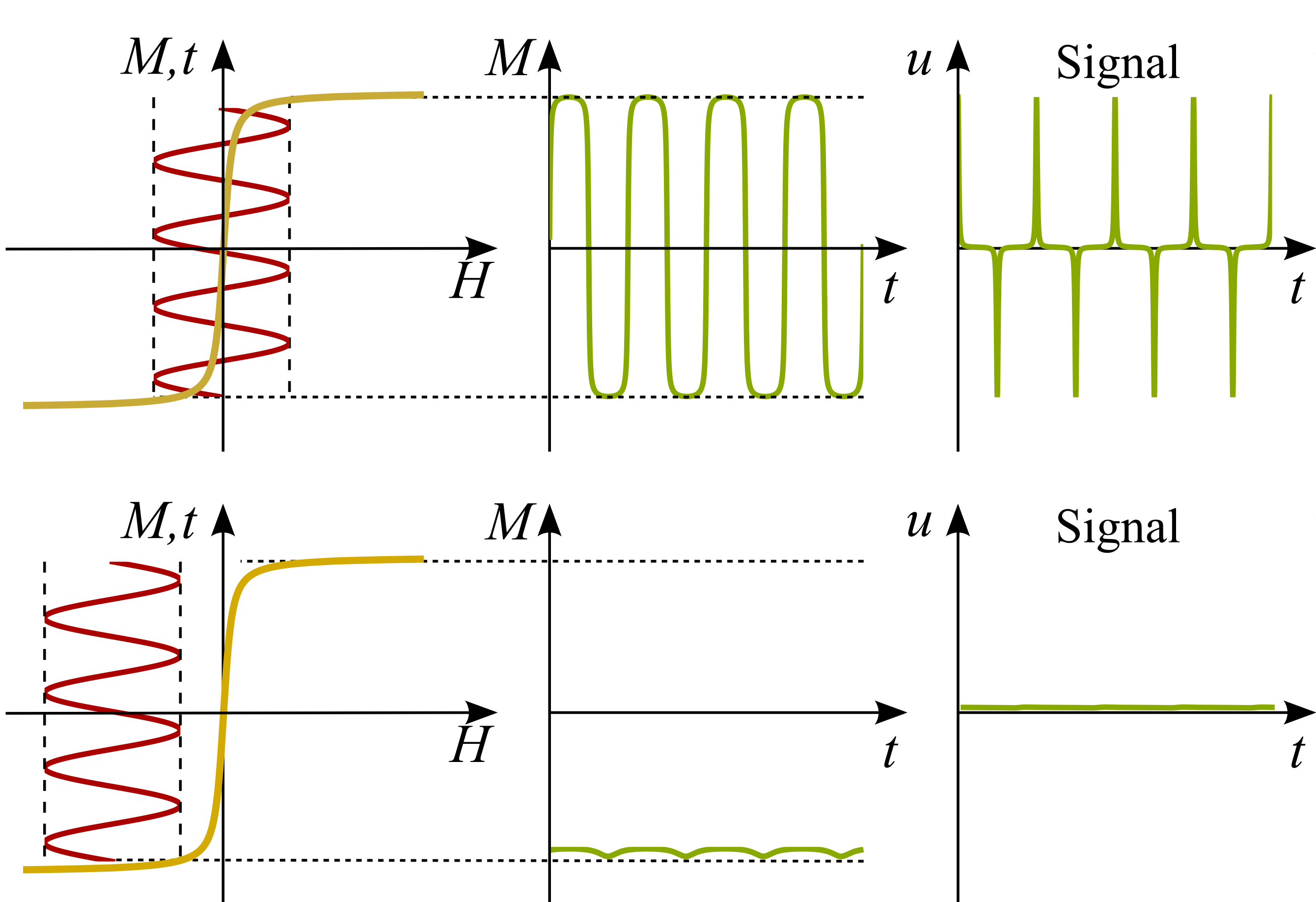}
\caption{1D principle of the non-linear magnetization response of superparamagnetic nanoparticles (courtesy of \cite{KnoppDiss2010}).
Left column: oscillations of $H(t,x)$ as function of time about a reference value $\hat{H}$ in the domain of the sigmoidal Langevin function $\Lan$ (yellow curve).
Middle column: the magnetization response. Right column: the received signal.
Oscillations about $\hat{H} \approx 0$ cause a non-trivial received signal, while oscillations about $|\hat{H}| \gg 0$ cause a zero signal.}
\label{fig:Response}
\end{center}
\end{figure}

\subsection{The MPI Signal Generation in Higher Dimensions}
\label{subsec:MPIForwHighdim}

Let us now consider the general multivariate situation. 
In contrast to the univariate setting where one scans along a straight line, 
it becomes important for the multivariate setting
that the signal is obtained by scanning along a trajectory which covers an area or a volume.

\paragraph{Scanning}
By design, the applied field $\vf{H}$ is the superposition of a static selection field $\vf{H}^S$ and a dynamic drive field $\vf{H}^D$:
\begin{align}
	\vf{H}(x,t) &= \vf{H}^S(x) + \vf{H}^D(x,t).
\end{align}
The static part $\vf{H}^S$ is generated by a direct current in the selection field coils.
The dynamic part is generated via an alternating current in the drive field coils and is described by
\begin{align}
	\vf{H}^D(x,t) &= \vf{P}(x) \vf{I}(t).
\end{align}
Here $\vf{I}(t)$ denotes the electric current in the coils and $\vf{P}(x)$ is the sensitivity profile of the drive field coils 
according to the Biot-Savart law \cite{chikazumi1978physics,jiles1998introduction}.

The locations where $\vf{H} \approx 0$ are important for the detection of particles (cf. Section~\ref{forwMPI1d}) which takes us
directly to the definition of the field free point (FFP) denoted by $r(t)$. 
The FFP is the geometric locus where  $\vf{H}$ vanishes:
\begin{align}\label{eq:equivFFP}
	\vf{H}(r(t),t) &= 0 \qquad \Leftrightarrow \qquad \vf{H}^S(r(t)) + \vf{P}(r(t)) \vf{I}(t) = 0.
\end{align}
From the second equality we learn that the current $\vf{I}(t)$ drives the FFP to scan the object.
Because the FFP trajectory $r(t)$ describes the locations where particles can be detected, the hull of $r(t)$
(under the assumption that $r(t)$ fills a portion of space very well)
defines the actual field of view.

\paragraph{Signal Encoding}
Schomberg \cite{Schomberg2010} points out that by solving the right-hand side of \eqref{eq:equivFFP} for $\vf{I}(t)$, 
$\vf{I}(t) = -\vf{P}(r(t))^{-1} \vf{H}^S(r(t))$, 
the applied field $\vf{H}$ can be described in terms of the FFP rather than the current:
\begin{align}
	\vf{H}(x,t) &= \vf{H}^S(x) - \vf{P}(x) \vf{P}(r(t))^{-1} \vf{H}^S(r(t)).
\end{align}
By design, the sensitivity pattern $\vf{P}$ is almost homogeneous in the FOV which simplifies the expression to
\begin{align}
	\vf{H}(x,t) &= \vf{H}^S(x) - \vf{H}^S(r(t)).
\end{align}
The static field $\vf{H}^S$ is generated by permanent magnets and/or a selection coil pair in Maxwell configuration \cite{chikazumi1978physics,jiles1998introduction}.  
It is therefore linear in $x$ with a constant gradient $\vf{G},$ i.e.,
\begin{align}\label{eqn:staticF}
	\vf{H}^S(x) &= \vf{G} x, & \vf{G} &= g\; \bnd{G}, &
	\bnd{G} &= \left(\begin{matrix}
               	-1 & 0 & 0 \\ 0 & -1 & 0 \\ 0 & 0 & 2
               	\end{matrix}\right).
\end{align}
The factor $g$ denotes the nominal gradient of the static field. 
Finally, the linearity of $\vf{H}^S$ implies that the applied field is of the form  \cite{Schomberg2010,GoodwillConolly2011}
\begin{align}\label{eqn:applH}
	\vf{H}(x,t) &= \vf{G}(x - r(t)), \quad \text{and also that} \quad \vf{I}(t) = - \vf{P}^{-1}\vf{G}r(t).
\end{align}

Again by design, the sensitivity pattern $\vf{R}$ of the recording coils is also almost homogeneous in the FOV.
By substituting \eqref{eqn:applH} into \eqref{eqn:magnet} and \eqref{eq:MPIproblemInTime}, the MPI signal is given by
\begin{align}\label{eqn:signal}
  \vf{s}(t) 
  &= \mu_0 \; m \; \vf{R} 
  \frac{d}{dt} \int\limits_{\R^3} \rho(x) \Lan \left( \frac{|\vf{G}(r(t)-x)|}{H_\sat} \right) \; \frac{\vf{G}(r(t)-x)}{|\vf{G}(r(t)-x)|} \; dx.
\end{align}

\subsection{Decomposition of the Signal Generation Process and Derivation of MPI Core Operator}\label{subsec:DecomposingMPIo}

Until now, the signal is described in terms of the trajectory of the field free point. Now, we see that actually less information is required; it turns out that only phase space information is needed meaning that the signal is uniquely determined by the location and the tangent to the trajectory. This reveals the core part of the MPI signal generation process, and makes the model independent of the particular trajectory under consideration.   
 
In order to extract the essential part in \eqref{eqn:signal}, we employ the spatial flux function $\vf{\Phi}$ defined by
\begin{align}\label{eqn:spatflux}
  \vf{\Phi}(r) &= \int\limits_{\R^3} \rho(x) \Lan \left( \frac{|\vf{G}(r-x)|}{H_\sat} \right) \; \frac{\vf{G}(r-x)}{|\vf{G}(r-x)|} \; dx.
\end{align}

Using $\vf{\Phi},$ we can rewrite \eqref{eqn:signal} as
\begin{align}
  \vf{s}(t) 
  &= \mu_0 \; m \; \vf{R} \frac{d}{dt} \vf{\Phi}(r(t)) = \mu_0 \; m \; \vf{R} \cdot \nabla_r \vf{\Phi}(r(t)) \cdot \dot{r}(t). 
\end{align}  
We derive the MPI core operator $\MPI[\rho]$ as the part which operates on $\rho,$ i.e.,  
\begin{align} \label{eq:MPIcoreOperator}
  \MPI[\rho](r) =
  \nabla_r \vf{\Phi}(r) &= \int\limits_{\R^3} \rho(x) \; \nabla_r \left(\Lan \left( \frac{|\vf{G}(r-x)|}{H_\sat} \right) \; \frac{\vf{G}(r-x)}{|\vf{G}(r-x)|}\right) \; dx.
\end{align}
We see that $\MPI[\rho]$ is a spatial convolution of $\rho$ with a matrix-valued kernel.
The convolution evaluated at the FFP $r(t)$ produces a matrix which is a applied to the velocity vector $v=\dot{r}(t)$ of the FFP. 
The central advantage of \eqref{eq:MPIcoreOperator} is that it is completely independent of a particular trajectory, it
describes MPI in space and gets rid of all decorative elements. Plugging in a trajectory $r(t)$ into \eqref{eq:MPIcoreOperator} the signal is obtained by 
\begin{align}
\vf{s}(t) &= \mu_0 \; m \; \vf{R} \cdot \MPI[\rho](r(t)) \cdot \dot{r}(t). 
\end{align}
A basic first step in this direction was already done by Goodwill et al. \cite{GoodwillConolly2011} 
in their \emph{$x$-space formulation} where they relate the signal description with spatial convolution. 

\paragraph{Scaling}
Now, we want to determine a single dimensionless parameter $h$ which characterizes the resolution of the MPI scanning setup.
Because the MPI core operator \eqref{eq:MPIcoreOperator}
depends on many different parameters, namely the nominal gradient $g$, the size $L$ of the field of view,
and a saturation parameter $H_\sat$ of the particles (which in turn is a function of the temperature $T$, the particle diameter $d$,
and the saturation magnetization $M_\sat$), we combine them in a single parameter $h$.
The smaller the parameter $h$ the better is the resolution. We will see here that in a typical setup $h$ is about $h \approx 10^{-2}$.
In Section~\ref{sec:AnaysisMPIop} which is concerned with the analysis of the MPI core operator we will study the idealization limit $h \to 0$
and its relation to the case of $h>0$.

The support of the particle concentration function $\rho$ is usually contained in the field of view $\Omega$ and of the same size, which determines the length scale.
As the field of view is a cube, we write it as a linear deformation of the standard cube $[-1,1]^3$
\begin{align}
 \Omega &= L \; \bnd{G}^{-1} \; [-1,1]^3,
\end{align}
with $\bnd{G}$ as of \eqref{eqn:staticF} and the length scale $L$ of the physical field of view. We use this now to transform variables in \eqref{eqn:signal} by
\begin{align}
 x &= L \; \bnd{G}^{-1} \; \nd{x} & r &= L \; \bnd{G}^{-1} \; \nd{r}.
\end{align}
That is $\nd{x}$ and $\nd{r}$ are elements of a non-dimensional field of view $\nd{\Omega}= [-1,1]^3$.
Moreover, we write the receive coil sensitivity as $R = \|R\|_2 \; \bnd{R}$ with a dimensionless matrix $\bnd{R}$ which gives only the receive pattern.
Equation \eqref{eqn:signal} transforms then to

\begin{align}\label{eqn:signal2}
  \vf{s}(t) &= \bnd{R} \frac{d}{dt} \int\limits_{\R^3}  \rd{\rho}(\nd{x}) \; 
  \Lan \left( \frac{ |\nd{r}(t)-\nd{x}|}{h} \right) \; \frac{\nd{r}(t)-\nd{x}}{|\nd{r}(t)-\nd{x}|} \; d\nd{x},
\end{align}
with
\begin{align}
  \rd{\rho}(\nd{x}) &:= c \; \rho(L \; \bnd{G}^{-1} \; \nd{x}), &
  c &:= \mu_0 \; m \; \frac{L^3}{\det \bnd{G}} \; \|R\|_2,  &
  h &:= \frac{H_\sat}{g \; L},
\end{align}
where $c$ is a dimensional constant, while $h$ is now a dimensionless shape parameter. 
By refering to the transformed spatial flux function
\begin{align}\label{eqn:spatflux2}
 \brd{\Phi}(\nd{r}) &:= \int\limits_{\R^3}  \rd{\rho}(\nd{x}) \; 
 \Lan \left( \frac{ |\nd{r}-\nd{x}|}{h} \right) \; \frac{\nd{r}-\nd{x}}{|\nd{r}-\nd{x}|} \; d\nd{x},
\end{align}
we have 
\begin{align}\label{eqn:signal3}
 \vf{s}(t) &= \bnd{R} \; \nabla_{\nd{r}} \brd{\Phi}(\nd{r}) \; \dot{\nd{r}}(t).
\end{align}

\begin{table}
	\centering
	\begin{tabular}{|c|c|c|c|}
		\hline
		param. & value & meaning & reference \\
		\hline
		$k_b$ & $1.38 \; 10^{-23} \; \frac{Nm}{K}$ & Boltzmann constant & \\
		$\mu_0$ & $4 \pi \; 10^{-7} \; \frac{N}{A^2}$ & vacuum magnetic permeability & \\
		$T$ & $310 K$ & temperature of human body & assumption \\
		$M_\sat$ & $0.6 \frac{T}{\mu_0}$ & saturation magnetization of particles & 
		cf. \cite{GoodwillConolly2010}\\
		$d$ & $20-30 \; 10^{-9} \;m$ & particle diameter & cf. \cite{BuzugKnopp2012} \\
		$g$ & $5.5 \; \frac{T}{\mu_0 \; m}$ & nominal gradient of static field & cf.  \cite{BuzugKnopp2012}\\
		$L$ & $20 \; 10^{-3} \; m$ & length of FOV & cf. \cite{BuzugKnopp2012}\\
		\hline
	\end{tabular}
	\caption{Parameters of the Langevin function from the literature.}\label{Tab:params}
\end{table}

\begin{figure}[t]
	\begin{center} 
		\includegraphics[width=0.70\textwidth]{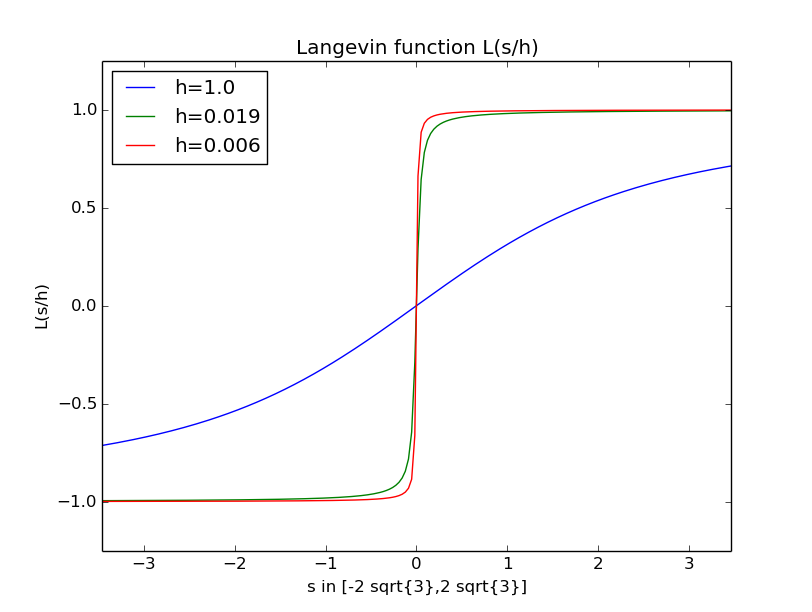}
		\caption{Langevin function $\Lan(s/h)$, $s \in [-2 \sqrt{3},2 \sqrt{3}]$ for different $h$.}
		\label{fig:Langevin}
	\end{center}
\end{figure}

The spatial resolution of an MPI scanner depends on how well the sigmoidal Langevin-function approximates the sign-function
on the relevant domain. By the non-dimensionalization above we obtain, with $s := |\nd{r}-\nd{x}|$, $s \in [0,2 \sqrt{3}]$, the diameter
of $\nd{\Omega}$ as the relevant domain of $\Lan(s/h)$. 
The size of $h$ determines the shape of $\Lan(s/h)$ on that domain. With smaller values for $h$ the Langevin-function gets closer to the sign-function. 
We have the following relationship between $h$ and the involved physical parameters:
\begin{align}
 h &= \frac{H_\sat}{g \; L}, & H_\sat &= \frac{k_b \; T}{\mu_0 \; M_\sat \; \frac{\pi}{6} d^3}.
\end{align}
$H_\sat$ is a combined parameter which describes 
the spherical superparamagnetic particles  
and is a function of the temperature $T$, the particle diameter $d$,
and the saturation magnetization $M_\sat$.
Table~\ref{Tab:params} lists values reported in the literature for the different parameters that determine $h$.
Based on these values we obtain with $20nm \leq d \leq 30nm$, the following range of $h$
\begin{align}
 0.019 \geq h \geq 0.006,
\end{align}
giving a quite good approximation of the sign-function, see Figure~\ref{fig:Langevin}.

\paragraph{Influence of the Electrical Current on the Scan Trajectory} 
The scan trajectory is controlled via the electrical current $\vf{I}(t)$. In dimensional form we have
\begin{align}\label{eqn:current}
 \vf{P} \; \vf{I}(t) &= -\vf{G} \;r(t).
\end{align}
By decomposing the current as well as the drive coil sensitivity into magnitude and direction, 
i.e. $\vf{I}(t) = I \; \bnd{I}(t)$ and $\vf{P} = \|\vf{P}\|_2 \; \bnd{P}$ we obtain 
a dimensionless form of \eqref{eqn:current}:
\begin{align}
 I \; \|\vf{P}\|_2 \; \bnd{P} \; \bnd{I}(t) &= - g\; L \; \nd{r}(t).
\end{align}
That means in order to scan a FOV of length scale $L$ we would need a magnitude of 
\begin{align}
 I &= -\frac{g}{\|\vf{P}\|_2} \; L
\end{align}
for the electrical current (assuming that $\bnd{P}$ does not shear too much).
Given that, then, up to a linear transformation, the dimensionless scan trajectory
\begin{align}\label{eqn:currentnondim}
 \nd{r}(t) &= \bnd{P} \; \bnd{I}(t)
\end{align}
inherits the signal shape of $\bnd{I}(t)$.

\paragraph{Lower Dimensional Particle Distributions} 
So far we have considered the signal generation and the MPI core operator in 3D.
But for lower dimensional particle distributions we can employ lower dimensional models, i.e. 1D or 2D.
Here, we provide a common canonical form valid for all three cases. 

We note that the physics of MPI is intrinsically 3D. Hence, even if the particle distribution is lower dimensional,
we initially have to start from a 3D setup.
Let us first assume that we are given a one-dimensional particle distribution 
$\rd{\rho}(\nd{x}) = \rho(\nd{x}_1) \; \delta(\nd{x}_2) \; \delta(\nd{x}_3)$. Here, one
would use a single pair of recording coils, i.e. $\bnd{R} = e_1^T$, and a single pair of drive coils, i.e. $\nd{r} = \nd{r}_1 e_1$.
In that case we have that $\vf{s}_1(t) = \frac{d}{d\nd{r}_1} \brd{\Phi}_1(\nd{r}_1) \; \dot{\nd{r}}_1(t),$ where
\begin{align}\label{eqn:spatflux3}
\brd{\Phi}_1(\nd{r}_1) &:= \bnd{R} \brd{\Phi}(\nd{r}) = \int\limits_{\R}  \rd{\rho}(\nd{x}_1) \; 
\Lan \left( \frac{|\nd{r}_1-\nd{x}_1|}{h} \right) \; \sign (\nd{r}_1-\nd{x}_1) \; d\nd{x}_1.
\end{align}
Accordingly we proceed in the case of two-dimensional particle distributions $\rd{\rho}(\nd{x}) = \rho(\nd{x}_1,\nd{x}_2)$ $\delta(\nd{x}_3)$.
In summary, we get the canonical form 
\begin{align}\label{eqn:common}
\brd{\Phi}(\nd{r}) &:= \int\limits_{\R^n}  \rd{\rho}(\nd{x}) \; 
\Lan \left( \frac{|\nd{r}-\nd{x}|}{h} \right) \; \frac{\nd{r}-\nd{x}}{|\nd{r}-\nd{x}|} \; d\nd{x}, &
\vf{s}(t) &= \bnd{R} \; \nabla_{\nd{r}} \brd{\Phi}(\nd{r}) \; \dot{\nd{r}}(t),
\end{align}
with a full rank matrix $\bnd{R} \in \R^{n \times n}$ and $\nd{r} \in [-1,1]^n$, which is valid for any dimension $n \in \{1,2,3\}$.

\subsection{Relation to the Previously Proposed Models}
\label{subsec:RelExistingModels}
Rahmer et al. \cite{Rahemeretal2009} consider a 1D scenario and study the Fourier series of the signal derived from a delta distribution. 
In the idealized 1D case, they derive a representation of the convolution kernel in terms of Chebyshev polynomials.
In the 2D and 3D scenario they study the Fourier transform of the signal obtained from inserting the Lissajou trajectory into the MPI core operator.
But a similar representation of the convolution kernel as in 1D is not available, see \cite{Rahmer_etal2012}.  

In contrast to \cite{Rahemeretal2009}, Goodwill and Conolly \cite{GoodwillConolly2010, GoodwillConolly2011} 
study the received signal in terms of the trajectory of the field free point.
In \cite{GoodwillConolly2010, GoodwillConolly2011} no Fourier transforms are considered and a general trajectory is inserted into the MPI core operator.
For the 1D case, it was shown in \cite{Gruettner_etal2013} that the x-space formulation is mathematically equivalent 
to the frequency space formulation of \cite{Rahemeretal2009}. 

The approach of Schomberg \cite{Schomberg2010} studies the integral of the received signal instead of the signal itself. 
The result is a formulation of convolution type involving the spatial flux $\vf{\Phi}$ of the MPI core operator. 

Summing up, all previous approaches are based on studying the concatenation of the MPI core operator \eqref{eq:MPIcoreOperator}
or the related flux \eqref{eqn:spatflux}
with some individually different operators.

\section{Analysis of the MPI Core Operator}\label{sec:AnaysisMPIop}
As seen in the previous section, the central task in analyzing MPI is to study 
the MPI core operator $A_h[\rho]$ which operates on the particle density $\rho$ .
The topic of this section is to start this investigation.
The MPI core operator is the derivative of the spatial flux function $\vf{\Phi}$.
The starting point for our analysis is \eqref{eqn:common}, where we from now on omit hats and bars in the notation.
Moreover, to have a handy notation, we consider the signal multiplied by $\bnd{R}^{-1}$ and call it $\vf{s}$ again.

Precisely, in this section, we consider 
\begin{align}\label{eqn:common2}
 \vf{\Phi}(r) &= \int\limits_{\R^n}  \rho(x) \; 
 \Lan \left( \frac{|r-x|}{h} \right) \; \frac{r-x}{|r-x|} \; dx, &
 \vf{s}(t) &= \nabla_{r} \vf{\Phi}(r) \; \dot{r}(t).
\end{align} 
Then the mathematical form of the MPI core operator is
\begin{align}\label{eqn:matfield}
A_h[\rho](r) =  \nabla_{r} \vf{\Phi}(r) = \int\limits_{\R^n}  \rho(x) \; 
\nabla_r \left( \Lan \left( \frac{|r-x|}{h} \right) \; \frac{r-x}{|r-x|} \right)\; dx. &
\end{align}
The operator $A_h$ acts on the density function $\rho$ and produces the matrix-valued function $A_h[\rho]$ which is a function of space $r$.

The time-dependent signal $\vf{s}(t)$ is obtained by inserting the 
scan trajectory in the spatial argument $r=r(t)$ and multiplying by the velocity vector $v= \dot{r}(t)$, i.e.,
\begin{align}\label{eq:MPImathProblemWithTime}
 \vf{s} &= A_h[\rho](r) \; v.
\end{align}

We see from \eqref{eqn:matfield} that the information on $\rho$ is encoded in the matrix-valued function  $A_h[\rho](r)$,
whereas the vectors $v$ in turn depend only on the employed scan trajectories which are independent of $\rho$.
This observation motivates the focus on the core operator $A_h$ and the relationship between $A_h[\rho]$ and $\rho$
which we investigate here.  
By scanning the object with different trajectories we can in principle gain full knowledge about $A_h[\rho](r)$;
practically, denser trajectories which visit each point several times allow for better approximations.

We start out in Section \ref{sec:idealCoreOp} to study the MPI core operator $A_h$ in the limiting case $h \to 0$. 
This reveals the central part of the core operator by freeing it from its blurring part; in particular, 
we observe dependence on the dimension and see that the previously conjectured Dirac property only holds in 1D but is no longer valid in a multivariate setting.
Then in Section~\ref{sec:nonIdealCoreOp}, we consider the non-idealized situation $h>0$ taking the blur
into account.

\subsection{The Idealized MPI Core Operator - the High Resolution Limit $h \to 0$ } \label{sec:idealCoreOp}

We learned in Section \ref{sec:MathematicalModel}  that the resolution parameter $h$ is of size $10^{-2}$.
Hence, the question of what happens in the ideal high resolution limit as $h \to 0$ is also interesting
and natural from an applied point of view. From a theoretical point of view, it turns out that
there are several relations between the idealized and the non-idealized MPI operator:
e.g., in the physically most important 3D case, the trace of  $A_h$ for a particular $h$ is just the concatenation of the trace in the idealization limit with a blurring operator.  
We will see this in Theorem~\ref{thm:RelaToLimit2} later on. 

Equation \eqref{eqn:matfield} tells us that $A_h$ is a convolution operator which operates linearly on $\rho$. 
In 1D, the kernel is scalar yielding a standard convolution operator and a standard deconvolution problem.
In contrast the multivariate setting comes with a matrix-valued kernel giving a non-standard convolution operator.
But, as can be seen in \eqref{eqn:matfield}, the components of the matrix-valued kernel act seperately on the scalar-valued $\rho$.
The striking point here is that the trace of $A_h[\rho]$ -- as we will see -- contains already all the information about $\rho$.
This has not been realized before but is the crucial step because
application of the matrix-trace to $A_h[\rho]$ gives a scalar deconvolution problem.
These facts allow for the derivation of reconstruction formulae. 
Moreover, we notice an interesting connnection to the fundamental solution of the Laplace equation.

\begin{theorem}\label{thm:IdealizationTheorem} 
The idealization limit $\alpha[\rho](r)= \lim_{h \to 0} \alpha_h[\rho](r)$ of the matrix-trace $\alpha_h[\rho](r) := \trace A_h[\rho](r)$ is of the form 
\begin{align}\label{eqn:tracelimit}
\alpha[\rho](r) &= \int\limits_{\R^n}   \rho(x) \; \kappa(r-x) \; dx.
\end{align}

In the multivariate setting, $n>1$, the limiting convolution kernel is
\begin{align}
\kappa(r-x) &:= \diver_r \left( \frac{r-x}{|r-x|} \right) = \frac{n-1}{|r-x|},
\end{align}
and \eqref{eqn:tracelimit} holds, for any $\rho$ in the space $\BV_0(\Omega)$ of functions of bounded variation with zero boundary, where $\Omega$ is the field of view. 

In the univariate setting, $n=1$, we have a Dirac-kernel
\begin{align}
\kappa(r-x) &:= 2 \; \delta(r-x),
\end{align}
and \eqref{eqn:tracelimit} holds for continuous functions $\rho$ with zero boundary.

Depending on the spatial dimension $n \in \{1,2,3\}$, we have three different relations of 
the convolution kernel to the fundamental solution $\Phi(r,x)$ of the Laplace equation, $-\laplace_r \Phi = \delta(r-x)$.
In 3D, we have 
\begin{align}\label{eq:kappa3D}
\kappa(r-x) = 8\pi \; \Phi(r,x)  \quad \text{with} \quad \Phi(r,x) = \frac{1}{4\pi \; |r-x|}.
\end{align}
In 2D, we have
\begin{align}\label{eq:kappa2D}
\kappa(r-x) = -2\pi \; \nabla_r \Phi(r,x) \cdot \frac{r-x}{|r-x|} 
 \quad \text{with} \quad \Phi(r,x) = -\frac{1}{2\pi}\log(|r-x|).
\end{align}
In 1D, we have
\begin{align}\label{eq:kappa1D}
\kappa(r-x) = -2 \frac{d^2}{dr^2} \Phi(r,x)  
\quad \text{with} \quad
\Phi(r,x) = -|r-x|/2.
\end{align}
	
\end{theorem}

Before proving this statement, we record some immediate consequences. The first consequence is the following reconstruction formula.

\begin{theorem}\label{thm:RecoIdeal} (Reconstruction Formula for the Idealized Case.)
Suppose that noise-free data $F(r) =  A_0[\rho](r)$ at each point $r$ is given for the idealized scenario, i.e. $A_0[\rho](r)= \lim_{h \to 0} A_h[\rho](r)$. 
Then,
\begin{align}
   \rho   =  \kappa^{-1} \circ \trace A_0[\rho].
\end{align}
That means, we take the pointwise trace and then deconvolve w.r.t. the  dimension-dependent $\kappa$ given by 
\eqref{eq:kappa3D},\eqref{eq:kappa2D}, or \eqref{eq:kappa1D}. 

\end{theorem}
\begin{proof}
 Given the matrix-valued data $F = A_0[\rho]$, we apply the matrix-trace. As a consequence of Theorem~\ref{thm:IdealizationTheorem} we obtain the deconvolution problem
 \begin{align}
  \trace F &= \kappa \ast \rho.
 \end{align}
 By inversion we obtain $\rho   =  \kappa^{-1} \circ  \trace F = \kappa^{-1} \circ \trace A_0[\rho]$.\hfill
\end{proof}

A next consequence is the following statement on ill-posedness in the sense of inverse problems 
\cite{Engl_Regularization_of_Inverse_Problems,kirsch2011introduction}. 

\begin{theorem}\label{thm:illposed} (Ill-Posedness.)
  Even the idealized MPI problem of Theorem~\ref{thm:RecoIdeal} 
  is ill-posed in 2D and 3D.
  (This is not the case in 1D where $\kappa$ is a Dirac distribution.)
  Depending on the dimension the degree of ill-posedness, i.e., the order of gained Sobolev smoothness of the forward operator,  
  is one in 2D, and two in 3D.
\end{theorem}

\begin{proof}
The degree of ill-posedness of the fundamental solution of the Laplace operator is two   
by the properties of the Laplace operator \cite{kirsch2011introduction}. 
Then, the statement is a consequence of \eqref{eq:kappa3D} and \eqref{eq:kappa2D} for the respective dimensions.\hfill
\end{proof}
	
Having recorded these consequences, we turn to the proof of Theorem~\ref{thm:IdealizationTheorem}.
\begin{proof}
We consider the MPI core operator $A_h[\rho]$ and notice that we can swap the gradient w.r.t $r$ for a gradient w.r.t. $x$ (giving a minus sign) to obtain
\begin{align}
  A_h[\rho](r) &= - \int\limits_{\R^n}  \rho(x) \; \nabla_x \left( \Lan \left( \frac{|r-x|}{h} \right) \; \frac{r-x}{|r-x|} \right)\; dx.
\end{align}
Then, the matrix-trace applied to $A_h[\rho]$ turns the gradient into the divergence operator
\begin{align}
 \alpha_h[\rho](r) &= \trace A_h[\rho](r) = - \int\limits_{\R^n}  \rho(x) \; \diver_x \left( \Lan \left( \frac{|r-x|}{h} \right) \; \frac{r-x}{|r-x|} \right)\; dx.
\end{align}

For the next steps we assume $\rho \in C_c^1(\Omega)$, which we relax later. By assumption, $\rho$ has compact support in the FOV, and we apply integration by parts to get 
\begin{align}
 \alpha_h[\rho](r) &:=  \int\limits_{\R^n}  \nabla_x \rho(x) \cdot \frac{r-x}{|r-x|} \; \Lan \left( \frac{|r-x|}{h} \right)\; dx.
\end{align}
As the integrand is dominated by $|\nabla_x \rho(x)|$, we obtain by Lebesgue's theorem on dominated convergence that
\begin{align}
 \lim\limits_{h \to 0} \alpha_h[\rho](r) &= \alpha[\rho](r) = \int\limits_{\R^n}  \nabla_x \rho(x) \cdot \frac{r-x}{|r-x|} \; dx.
\end{align}
Finally, we apply integration by parts again and undo the swap of derivatives which results in
\begin{align}\label{eqn:tracelimit2}
 \alpha[\rho](r) &= \int\limits_{\R^n}   \rho(x) \; \diver_r \left( \frac{r-x}{|r-x|} \right)\; dx.
\end{align}
We now distinguish between the multivariate and the univariate case.

\emph{The multivariate case, $n>1$.} The limiting convolution kernel can be expanded to
\begin{align}
  \kappa(r-x) &:= \diver_r \left( \frac{r-x}{|r-x|} \right) = \frac{n-1}{|r-x|}.
\end{align}
This may be seen by the following two steps: the Jacobian of the function $r/|r|$ is
\begin{align}
 D_r \frac{r}{|r|} &= \frac{1}{|r|} \left( I - \frac{r}{|r|} \frac{r^T}{|r|} \right),
\end{align}
where $I$ is the identity matrix. The divergence $\diver_r (r/|r|) = (n-1)/|r|$ is then given by the trace of the Jacobian.
The dependence on the spatial dimension $n$ is because of $\trace I = n$.
For $n>1$ the result \eqref{eqn:tracelimit2} holds also for functions $\rho \in \BV_0(\Omega)$ because these can be approximated by functions from $C_c^\infty(\Omega)$
in the strict topology on $\BV_0(\Omega);$ cf. \cite{ambrosio2000functions}.

\emph{The univariate case, $n=1$.} Here we have
\begin{align}
  \kappa(r-x) &:= 2 \; \delta(r-x).
\end{align}
This is because for $n=1$ the divergence expression comes down to
\begin{align}
 \diver_r \left( \frac{r-x}{|r-x|} \right) &= \frac{d}{dr}  \sign(r-x) = 2 \; \delta(r-x). 
\end{align}
(The last equality is to be understood in the sense of distributions.)
For $n=1$ the result \eqref{eqn:tracelimit2} holds also for continuous functions $\rho \in C_0(\Omega)$
because these can be approximated by functions from $C_c^\infty(\Omega)$ in the norm topology on $C_0(\Omega)$.

Finally, we relate the convolution kernels to the fundamental solution of the Laplace equation
\begin{align}
 -\laplace_r \Phi &= \delta(r-x)
\end{align}
Depending on the spatial dimension $n \in \{1,2,3\}$ we have three different scenarios:
\begin{enumerate}[a)]
 \item In 3D, $n=3$, the fundamental solution is given by $\Phi(r,x) = \frac{1}{4\pi \; |r-x|}$,
 while the convolution kernel is
 \begin{align}
  \kappa(r-x) &= \frac{2}{|r-x|} = \frac{8\pi}{4\pi\; |r-x|} = 8\pi \; \Phi(r,x).
 \end{align}
 \smallskip
 \item In 2D, $n=2$, the fundamental solution is given by $\Phi(r,x) = -\frac{1}{2\pi}\log(|r-x|)$ and its gradient w.r.t $r$ is
 \begin{align}
  \nabla_r \Phi(r,x) &= -\frac{1}{2\pi \; |r-x|} \frac{r-x}{|r-x|}, &  \nabla_r \Phi(r,x) \cdot \frac{r-x}{|r-x|} &= -\frac{1}{2\pi \; |r-x|}.
 \end{align}
 The convolution kernel is
 \begin{align}
  \kappa(r-x) &= \frac{1}{|r-x|} = -\frac{-2\pi}{2\pi\; |r-x|} = -2\pi \; \nabla_r \Phi(r,x) \cdot \frac{r-x}{|r-x|}.
 \end{align}
 \smallskip
 \item In 1D, $n=1$, the fundamental solution is given by $\Phi(r,x) = -|r-x|/2$.
 The convolution kernel is a Dirac and thus related to the second derivative of the fundamental solution:
 \begin{align}
  \kappa(r-x) &= 2 \; \delta(r-x) = -2 \; \frac{d^2}{dr^2} \Phi(r,x).
 \end{align}
\end{enumerate}
This was the last assertion of the theorem and completes the proof.\hfill
\end{proof}

The Dirac property in 1D has already been reported by Gleich and Weizenecker in \cite{GleichWeizenecker2005}, and was until now assumed to hold in higher dimensions, too.
The news is that the convolution kernels with parameter $h$ in higher dimensions are not Dirac sequences.
In contrast, we end up with non-trivial convolution kernels which are lower-order derivatives of the fundamental solution of the Laplacian.
This makes sophisticated deconvolution techniques necessary even in the ideal high resolution limit, as we have seen in Theorem~\ref{thm:illposed}.

\subsection{The MPI Core Operator - the Non-Ideal Case $h > 0$} 
\label{sec:nonIdealCoreOp} 

In concrete applications, it is, for sure, important to reconstruct in the non-idealized case $h>0$. 
Loosely speaking, it will turn out that a non-zero resolution parameter $h>0$ results in blured data.

We start out by deriving a statement analogous to the first part of Theorem~\ref{thm:IdealizationTheorem}
which shows, that, also in the non-idealized case, the trace  $\alpha_h[\rho]$ 
of the MPI core  operator can be written as a convolution operator with a scalar-valued convolution kernel $\kappa_h$. 
In particular, we see that this convolution kernel is an analytic radially symmetric function.
The analyticity makes the reconstruction problem even severely ill-posed from an inverse problem point of view. 

\begin{theorem}\label{thm:NonIdealizationTheorem}
  The trace $\alpha_h[\rho]$  of the MPI core operator $A_h[\rho]$ is given by 	
  \begin{align}\label{eqn:nonidealConv}
    \alpha_h[\rho](r) &= \int\limits_{\R^n}   \rho(x) \; \kappa_h(r-x)\; dx,
  \end{align}
  where the convolution kernel $\kappa_h$ takes the form
  \begin{align} \label{eq:defkappah}
    \kappa_h(y)&= \frac{1}{h} \, f\left( \frac{|y|}{h} \right), \qquad \text{ with } \quad f(z) = \Lan'(z) + \Lan(z) \frac{n-1}{z},
  \end{align}
  where $\Lan$ is the Langevin function. In particular, $f$ is analytic with singularities only on the imaginary axis.
  Near zero, $f$ has the power series expansion 
  \begin{align}
    f(z) &= \sum\limits_{k=0}^{\infty} \frac{2^{2k+2}B_{2k+2}}{(2k+2)!} \; (2k+n) \; z^{2k},
  \end{align}
  with a convergence radius of $\pi$. ($B_{l}$ denotes the $l$-th Bernoulli number.)
  Because of this the kernel $\kappa_h$ is itself analytic and \eqref{eqn:nonidealConv} holds for any $\rho \in \BV_0(\Omega)$ independently of the spatial
  dimenion $n \in \{1,2,3\}$.
\end{theorem}

We postpone the proof of this statement to the end of this section and record some immediate consequences. 
The first consequence is the following reconstruction formula.

\begin{theorem}\label{thm:RecoNonIdeal} (Reconstruction Formula for the Non-Idealized Case.)
Suppose that noise-free data $F(r) =  A_h[\rho](r)$ at each point $r$ is given. 
Then,
\begin{align}
   \rho   =  \kappa_h^{-1} \circ \trace A_h[\rho].
\end{align}
That means, we take the pointwise trace and then deconvolve w.r.t. $\kappa_h$ given by \eqref{eq:defkappah}.
The inverse of the convolution operator $\kappa_h$ is denoted by $\kappa_h^{-1}$.
\end{theorem}
\begin{proof}
 Given the matrix-valued data $F = A_h[\rho]$, we apply the matrix-trace. 
 As a consequence of Theorem~\ref{thm:NonIdealizationTheorem} we obtain the deconvolution problem
 \begin{align}
  \trace F &= \kappa_h \ast \rho.
 \end{align}
 By inversion we obtain $\rho   =  \kappa_h^{-1} \circ  \trace F = \kappa_h^{-1} \circ \trace A_h[\rho]$.\hfill
\end{proof}

A next consequence is the following statement giving a first relationship between the convolution kernel $\kappa_h$ and the idealization limit $\kappa$.
It reveals again the dimension dependence from a different viewpoint.
\begin{corollary}\label{cor:RelaToLimit1} (Relation to the Idealization Limit I.)
The kernels $\kappa,\kappa_h$ representing the traces of the idealized and the non-idealized MPI operator are related by 
\begin{align}\label{eqn:rela}
\kappa_h(y)&= \Lan' \left( \frac{|y|}{h} \right) \; \frac{1}{h} + \Lan \left( \frac{|y|}{h} \right) \; \frac{n-1}{|y|},
\end{align}
where $\Lan$ is the Langevin function.

In the univariate case, $n=1$, the second summand of \eqref{eqn:rela} vanishes, while the first summand is a Dirac-sequence:
\begin{align}
 \kappa_h(y)&= \Lan' \left( \frac{|y|}{h} \right) \; \frac{1}{h} \quad \to \quad 2 \; \delta(y) = \kappa(y) \quad \text{ as } \quad h \to 0.
\end{align}

In the multivariate case, $n>1$, the factor $(n-1)/|y|$ of the second summand is already the idealized convolution kernel $\kappa(h)$.
The first summand is a zero-sequence, while the second summand tends to $\kappa(y) = (n-1)/|y|$.
\begin{align}
 \kappa_h(y)&= \Lan' \left( \frac{|y|}{h} \right) \; \frac{1}{h} + \Lan \left( \frac{|y|}{h} \right) \; \frac{n-1}{|y|}
 \quad \to \quad 0 + \frac{n-1}{|y|} = \kappa(y) \quad \text{ as } \quad h \to 0.
\end{align}
The limits here are understood in the sense of distributions.
\end{corollary}

\begin{proof}
This is an immediate consequence of Theorems~\ref{thm:NonIdealizationTheorem} and \ref{thm:IdealizationTheorem}:
Theorem~\ref{thm:NonIdealizationTheorem} yields the trace of the MPI core operator in terms of a convolution $\alpha_h[\rho] = \kappa_h \ast \rho$ with
\begin{align}\label{eqn:kappah}
\kappa_h(r-x)&= \Lan' \left( \frac{|r-x|}{h} \right) \; \frac{1}{h} + \Lan \left( \frac{|r-x|}{h} \right) \; \frac{n-1}{|r-x|}.
\end{align}
Theorem~\ref{thm:IdealizationTheorem} gives the limit
\begin{align}
 \kappa_h \ast \rho(r) &= \alpha_h[\rho](r) \quad \to \quad \alpha[\rho](r) = \kappa \ast \rho(r) \quad \text{ as } \quad h \to 0.
\end{align}
In the univariate case, $n=1$, the second summand of \eqref{eqn:kappah} vanishes. The limiting kernel is $\kappa(r-x) = 2\; \delta(r-x)$, 
thus the first summand is a Dirac-sequence.

In the multivariate case the factor $(n-1)/|r-x|$ is the limiting convolution kernel while the factor $\Lan \left( |r-x|/h \right)$
tends to one almost everywhere, thus the first summand is a zero-sequence. \hfill
\end{proof}

A particularly nice relation between the traces $\alpha[\rho]$ and $\alpha_h[\rho]$ of the idealized and the non-idealized MPI operators is the
already announced connection in terms of convolution. It tells us that in 3D the non-idealized $\alpha_h[\rho]$ is a smoothed version of the idealized 
$\alpha[\rho]$.

\begin{theorem}\label{thm:RelaToLimit2} (Relation to the Idealization Limit II.)
In the three-dimensional case, $n=3$, we have the following relationship between the non-idealized $\alpha_h[\rho]$ and the idealized $\alpha[\rho]$:
\begin{align}\label{eqn:Rela3D}
\alpha_h[\rho] = -\frac{\laplace \kappa_h}{8\pi} * \alpha[\rho].
\end{align}	
Because of the analyticity of $\kappa_h$, $\alpha_h[\rho]$ is an infinitely smoothed version of the idealized $\alpha[\rho]$.
\end{theorem}
\begin{proof}
By Theorem~\ref{thm:NonIdealizationTheorem} we can express $\alpha_h[\rho]$ by convolution. Here, we supplement a Dirac-$\delta$ and get
\begin{align}
 \alpha_h[\rho] &= \kappa_h * \rho = (\kappa_h * \delta) * \rho.
\end{align}
In 3D, we can replace the Dirac-$\delta$ with the negative Laplacian of the fundamental solution. In addition we move the Laplace operator over to $\kappa_h$ to obtain
\begin{align}
 \alpha_h[\rho] &= (\kappa_h * (-\laplace \Phi)) * \rho = (-\laplace \kappa_h *  \Phi) * \rho.
\end{align}
By Theorem~\ref{thm:IdealizationTheorem}, we have $\kappa = 8\pi \; \Phi$. Now we express $\Phi$ in terms of the idealized kernel $\kappa$ which results in
\begin{align}
 \alpha_h[\rho] &= \left(-\laplace \kappa_h *  \frac{1}{8\pi} \kappa \right) * \rho = -\frac{\laplace \kappa_h}{8\pi} * ( \kappa * \rho).
\end{align}
Finally, we use the result of Theorem~\ref{thm:IdealizationTheorem} stating that $\kappa \ast \rho = \alpha[\rho]$, which yields equation \eqref{eqn:Rela3D}. \hfill
\end{proof}

Furthermore we obtain the following statement on severe ill-posedness in the sense of inverse problems
\cite{Engl_Regularization_of_Inverse_Problems,kirsch2011introduction}. 
\begin{theorem} (Severe Ill-Posedness.)\label{thm:SevIllposed}
  The non-idealized MPI problem of Theorem~\ref{thm:NonIdealizationTheorem} is severely ill-posed in 
  the following sense: there are no two spaces $H^s,H^t$ in the Sobolev scale, such that the  
  trace $\alpha_h$ of the MPI operator induces an isomorphism $\alpha_h:H^s \to H^t$ between these two spaces.
\end{theorem}
\begin{proof}
  This is an immediate consequence of the analyticity of the kernel $\kappa_h$ and the fact that an isomorphism has a closed image.
  From Theorem~\ref{thm:RelaToLimit2} we learn that $\alpha_h[\rho]$ is analytic for any $\rho$; hence the image of $\alpha_h$ is necessarily not closed in $H^t$.\hfill
\end{proof}

Finally, we supply the proof of Theorem~\ref{thm:NonIdealizationTheorem}.
\begin{proof}
By applying the matrix-trace to $A_h[\rho]$ we obtain
\begin{align}
 \alpha_h[\rho](r) &= \trace A_h[\rho](r) = \int\limits_{\R^n} \rho(x) \; \diver_r \left( \Lan \left( \frac{|r-x|}{h} \right) \frac{r-x}{|r-x|} \right) \; dx.
\end{align}
The convolution kernel in the non-idealized case is thus given by
\begin{align}
 \kappa_h(r-x)&:= \diver_r \left( \Lan \left( \frac{|r-x|}{h} \right) \frac{r-x}{|r-x|} \right).
\end{align}
By employing the product rule for the divergence, the kernel can be expanded to
\begin{align} 
 \kappa_h(r-x) &= \nabla_r \left( \Lan \left( \frac{|r-x|}{h} \right) \right) \cdot \frac{r-x}{|r-x|} 
 +  \Lan \left( \frac{|r-x|}{h} \right) \diver_r \left( \frac{r-x}{|r-x|} \right) \\
 & = \Lan' \left( \frac{|r-x|}{h} \right) \; \frac{1}{h} + \Lan \left( \frac{|r-x|}{h} \right) \; \frac{n-1}{|r-x|}.  
\end{align}
We express the kernel via the scalar function $f$ as
\begin{align}
 \kappa_h(r-x)&= \frac{1}{h} \; f\left( \frac{|r-x|}{h} \right) & f(z) &= \Lan'(z) + \Lan(z) \frac{n-1}{z}.
\end{align}
From the Laurent series of the $\coth$ we learn that the Langevin function $\Lan$ has the power series 
\begin{align}
 \Lan(z) &= \coth(z) - \frac{1}{z} = \sum\limits_{k=1}^{\infty} a_k \; z^{2k-1}, & a_k &= \frac{2^{2k}B_{2k}}{(2k)!}, 
\end{align}
where $B_{2k}$ are the Bernoulli numbers. It is analytic at zero with radius of convergence $\pi$. 
Consequently, the function $f$ is also analytic near zero; it can be written as 
\begin{align}
 f(z) &= \sum\limits_{k=0}^{\infty} a_{k+1} \; (2k+n) \; z^{2k}.
\end{align}
All further singularities of $\coth(z)$ (besides the one at $0$ which was involved above) are  nonzero and lie on the imaginary axis. This implies that the singularities of $f$ are
nonzero and lie on the imaginary axis as well. 
Using $y = r-x$ as short hand we see that the convolution kernel $\kappa_h(y)$ is analytic near $y=0$ by
\begin{align}
 \kappa_h(y) &= \frac{1}{h} \sum\limits_{k=0}^{\infty} a_{k+1} \; (2k+n) \; \left( \frac{|y|^2}{h} \right)^{k}.
\end{align}
And away from $y=0$ it is concatenation of analytic functions.
\hfill 
\end{proof}

\section{A Model-Based Reconstruction Algorithm for Multivariate MPI}
\label{sec:ReconstructionAlgo}

In this section, we apply the derived reconstruction formulae for the MPI core operator to establish a reconstruction algorithm for MPI in the multivariate (2D and 3D) cases.    
We start with the discretization in Section~\ref{subsec:disc}. 
In particular, we recast the discrete time data into spatial 
data and obtain a reconstruction problem for the discretized MPI core operator.
Then we present our reconstruction algorithm in Section~\ref{subsec:recAlgoDesc}.
Finally, we illustrate our algorithm with a numerical example in Section~\ref{subsec:numExperiment}.

\subsection{Discretization and Reduction to a Reconstruction Problem for the Discretized MPI Core Operator} \label{subsec:disc}

As described in Section~\ref{sec:MathematicalModel}, the measured data is the time-dependent data $\vf{s}(t)$ which is associated with a scan trajectory $r(t)$.
The scan trajectory inherits the shape of the electrical current $\vf{I}(t)$ given in \eqref{eqn:currentnondim} by 
\begin{align}
r(t) &= \vf{P} \; \vf{I}(t).
\end{align}
(Here we refer to dimensionless quantities, but omit hats.)
One advantage of the present approach is that it is independent of the particular trajectory type employed. Furthermore, in our approach, even different trajectories might be combined.
The important point is that the data is given as living on a discrete sampling of the trajectory $r(t),$
i.e., we observe values $\vf{s}(t_k)$ at location $r(t_k),$ with the trajectory having tangent $v(t_k)$
at time $t_k,$ for finitely many measurements indexed by $k$.  

These data are a discrete sampling of the MPI core operator 
$\MPI(r_k,v_k) $ $= A_h[\rho](r_k)v_k.$
Our first task is to employ these data to reconstruct a discrete version of the trace $\alpha_h[\rho].$
To reconstruct these trace data in a discrete setup, 
we consider discrete functions defined on a grid of $N_1 \times N_2$ cells.
(We explain the 2D setting in the following; the 3D setting is obtained by straight forward modifications.)
On each cell a function is constant and each cell is represented by its cell center point.
We assume that the cells (corresponding to the chosen spatial resolution) are chosen in a way such that each cell of the grid is being passed at least twice from different directions with the corresonding 
trajectory points $r(t_k)$ hitting the cell.
In the following, we say that a time sample $t_k$ belongs to cell $i,$ if $r(t_k)$ is in this cell.
For each cell $i$ we collect the signal data $\vf{s}(t_{k_i})$ from those times samples $t_{k_i},$
such that $t_{k_i}$ belongs to the cell $I$ (i.e., $r(t_{k_i})$ hits the cell $i$), and gather them in a matrix $S_i$.
Accordingly, we collect all the velocity vectors $\dot{r}(t_{k_i})=v(t_{k_i})$ and gather them in a matrix $V_i$. 
Associating all points that fall in the same cell with the cell center point $x_i$,
and then using the signal encoding model, we have for each cell $i,$ the following matrix fitting problem w.r.t.\ $A_i,$
\begin{align}\label{eqn:LS}
	A_i \; V_i &= A_h[\rho](x_i) \; V_i = S_i.
\end{align}
By our assumptions above it is ensured that $V_i$ has full rank. 
(For sure, it is possible to relax the above assumption considerably by involving spatial regularity assumptions; see our discussion on future work.)
We solve the system \eqref{eqn:LS} by least squares fitting : we let $V_i^T = Q_i \; R_i$ be the reduced $QR$-decomposition \cite{golub2012matrix} with a full rank $2\times 2$ matrix $R_i$.
We obtain $A_i$ and the corresponding trace data $u_i$ on cell $i$ (corresponding to $\alpha_h[\rho](x_i)$) by
\begin{align}\label{eq:preprocessing2}
	A_i = S_i \; Q_i \; R_i^{-T},\qquad \text{and} \qquad  u_i &= \trace A_i.
\end{align}

\subsection{Reconstruction Algorithm} \label{subsec:recAlgoDesc}

Our reconstruction scheme may be subdivided into two major steps: 
in the first step we produce spatial data from the given time data on the trajectory by \eqref{eq:preprocessing2}.
This is an intermediate step which yields the trace data
and represents the particle distribution $\rho$ convolved with the kernel $\kappa_h$ according to \eqref{eq:defkappah}.
In the second step we perform deconvolution, i.e. inversion of $\kappa_h$,       
which we do in a regularized way since we learned in Theorem~\ref{thm:SevIllposed} that this problem 
is ill-posed. In this second step, we employ the analytic results of Section~\ref{sec:AnaysisMPIop}.

\paragraph{Step 1: Deriving Trace Data on a Spacial Grid from the Raw Input}
We obtain a grid function representing trace data $u_i$ in each pixel (grid cell) as explained in Section~\ref{subsec:disc}.

\paragraph{Step 2: Reconstruction of the Signal from the Derived Trace Data by Deconvolution}
Having obtained data $u_i$ related to the trace $\alpha_h[\rho](x_i)$ from Step 1,
we now take the convolution description of $\alpha_h[\rho](r)$ derived in 
Section~\ref{sec:AnaysisMPIop} into account: if the data were noiseless, we would
have to solve the problem of finding $\rho$ in $\kappa_h * \rho = u$ given $u.$
However, normally the data $\vf{s},$ and therefore also $u,$ are noise-contaminated, 
and we have learned in Section~\ref{sec:AnaysisMPIop}
that this reconstruction problem is severely ill-posed. Therefore, regularization is needed; cf. 
\cite{Engl_Regularization_of_Inverse_Problems,louis1989inverse,bertero1998introduction,kirsch2011introduction}.
As regularization technique, we here use a variational approach based on classical 
Tychonov regularization. More precisely, we consider the following (spatially continuous) problem:
\begin{align}
\rho &= \arg \min\limits_{\hat{\rho}} \; \mu \; \|\nabla_r \hat{\rho}\|_{L^2}^2 + \| \kappa_h * \hat{\rho} - u \|_{L^2}^2.
\end{align}

The convolution with $\kappa_h$ is discretized on the considered $N_1 \times N_2$ grid using the summed-up midpoint rule. We denote the resulting matrix by $K_h$.
Moreover, we discretize the gradient with forward finite differences to obtain a corresponding matrix $D$.
The discrete Tychonov-regularized problem is now of the form 
\begin{align} \label{eq:DeconvTych}
 \rho &= \arg \min\limits_{\hat{\rho}} \; \mu \; \|D \hat{\rho}\|_{2}^2 + \| K_h \; \hat{\rho} - u \|_{2}^2.
\end{align}
with vectorized $\rho,u \in \R^{N_1 \cdot N_2}$.
We solve the corresponding discrete Euler-Lagrange equation
\begin{align}\label{eqn:EulerLagrange}
 -\mu \; L \rho + K_h \left(K_h\rho - u\right) &= 0.
\end{align}
Here, $-L = D^T D$ is the five point stencil discretization of the Laplacian with zero Dirichlet boundary and $K_h=K_h^T$ is symmetric.  
The whole reconstruction algorithm is subsumed as Algorithm~\ref{alg:recAlgo}.

\begin{algorithm}[t]
\SetKwFunction{CompLambda}{CompLambda}
\SetKwFunction{CompData}{GeoLengthData}
\SetKwFunction{CompReg}{GeoLengthData}
\SetKwFunction{CompGradient}{CompGradient}
\SetKwFunction{CompGradientT}{CompGradientT}
\SetKwFunction{QR}{QR}
\SetKwFunction{searchLine}{searchLine}
\SetKwFunction{CompGradientTprime}{CompGradientTprime}
\SetKwInOut{Input}{input}
\SetKwInOut{Output}{output}
\BlankLine
\Input{Time dependent samples $s_k=s(t_k)$ along trajectory $r_k=r(t_k)$ with tangent 
	$v_k= \dot{r}(t_k)$  at times $t_k;$ regularization parameter $\mu$.  }
\Output{Reconstructed particle density $\rho$ (solution of \eqref{eq:DeconvTych} with preprocessing given by \eqref{eq:preprocessing2},\eqref{eqn:LS}).}
\BlankLine
Initialize $\rho$\;
\BlankLine

\For{$k\leftarrow 1$ \KwTo $K$}
{
	
	\BlankLine
	
	\texttt{// Collect data:}\\
	\texttt{// associate time samples with pixel grid; cf.~Section~\ref{subsec:disc}.}
	
	\BlankLine	
	
	\If{$r_k$ \text{in cell} $i$}
	{
	  $
	  \begin{aligned}
	   V(i)   &\leftarrow [V(i), \  v_k]; && \texttt{// Append tangent direction.} \\
	   S(i)   &\leftarrow [S(i), \  s_k]; && \texttt{// Append data value.}
	  \end{aligned}
	  $
	}
}

\BlankLine

\For{$i\leftarrow 1$ \KwTo $I$}
{
	\BlankLine
	\texttt{// For each cell fit trace data using~\eqref{eqn:LS} and \eqref{eq:preprocessing2}.}
	
	\BlankLine
	
	$ [Q_i,R_i] \leftarrow \QR(V^T_i)  $; \quad 
	
	$ A_i \leftarrow S_i \ Q_i \ R_i^T$; \quad 
	
	$u_i \leftarrow \trace A_i$;
		
}

\BlankLine
\texttt{// Regularized deconvolution of the trace data using \eqref{eq:DeconvTych} by}\\
\texttt{// solving \eqref{eqn:EulerLagrange} with conjugate gradients (CG).}\\
$ \rho = CG(-\mu \; L + K_h^2, \; K_h u)$; 
\BlankLine

\BlankLine
\caption{Proposed model-based reconstruction algorithm for the MPI problem; cf. \eqref{eq:MPIproblemInTime}.}
\label{alg:recAlgo}
\end{algorithm}
\DecMargin{1em}

\subsection{Numerical Results} \label{subsec:numExperiment}

We demonstrate the potential of our method by
applying the developed method to synthetic data. 
Detailed numerical investigations are the topic of a forthcoming paper.

\paragraph{Generation of Synthetic Data} 
We generate synthetic data 
using Lissajous curves \cite{lawrence2013catalog}. More precisely, we simulate the current signal $\vf{I}(t)$ as 
the Lissajous curve $l(t)$ which is defined by
\begin{align}
l(t) &= (\sin(2\pi \ m_1 t), \sin(2\pi \ m_2 t)).
\end{align}
Here $m_1,m_2$ are integer frequencies which implies that the trajectory is closed. 
Then, we assume for simplicity, that the sensitivity pattern of the drive coils $\vf{P}$ is the identity matrix, i.e.
the trajectory of our non-dimensional field-free point $r(t)$ is thus the Lissajous curve.
We note that the described framework does not depend on this particular choice of trajectories.

\paragraph{Experimental Setup}
For our experiment we used a square $N \times N = 100 \times 100$ grid,
as the underlying spatial signal --the ground truth-- we use Figure~\ref{fig:numericalRec} (d).

To simulate a scan we use the above Lissajous curves with frequencies $m_1=101$ and $m_2=102$.
We pick $K=20 \cdot N^2 = 200000$ equidistant time samples $t_k$ from the time interval $[0,1]$
to obtain spatial points $r_k=r(t_k)$ with tangents $v_k=v(t_k)$ on the Lissajous curve. 

Finally, we simulate the MPI time signal: we start with generating a signal according to the model 
\begin{align}
 s_k &= A_h[\rho](r_k) \; v_k,
\end{align}
where we evaluate $A_h[\rho](r_k)$ discretely by approximating the matrix-valued convolution via the summed-up midpoint rule on our $100 \times 100$ grid.
The resolution parameter $h$ of the convolution kernel is set to $h=10^{-2}$ and is in the range which was identified in Section~\ref{sec:MathematicalModel}.
Next, we add noise to imitate the noise-contamination of signal data from a MPI scanner
\begin{align}
 \hat{s}_k &= s_k + \varepsilon N_k.
\end{align}
The $N_k$ are i.i.d normally distributed with zero mean and a standard deviation of one.
The noise amplitude $\varepsilon$ is set to 10 percent of the signal strength:
\begin{align}
 \varepsilon &= 0.1 \; \max\limits_{k=1:K}\{|s_k|\}.
\end{align}
The perturbed signal $\hat{s}_k$ is now the input time data for our reconstruction algorithm. 

With this experimental setup we obtain the trace data $u_i$ from the signal $\hat{s}_k$.
The last step of the algorithm is regularized deconvolution. The convolution kernel $\kappa_h$ here has the same value for $h=10^{-2}$ as used in the setup.
The Tychonov regularization-parameter $\mu$ was set to $\mu = 3 \cdot 10^{-4}$.
Moreover, we used the CG-method from the Octave library with a relative tolerance $\tau$ of $\tau=2 \; 10^{-3}$
to solve \eqref{eqn:EulerLagrange} for $\rho$.
The last two parameters $\mu$ and $\tau$ have to be adjusted in general. 
With our choice we obtained good results for our test case and
the CG-method converges within 29 iterations.

Finally, the experiments were conducted on a laptop with a Intel Core I5-3337U CPU, 1.8 GHz, and 8GB of RAM.
The programs were implemented in Octave and run on Octave 3.8.1 under Ubuntu 14.04.

\paragraph{Discussion of the Results}
The result of our method is shown Figure~\ref{fig:numericalRec}.
The generated noisy time signal $\hat{s}$ is displayed Figure~\ref{fig:numericalRec}(a).
By \eqref{eq:preprocessing2} we obtain the spatial data $u_i = \alpha_h[\rho](x_i)$ out of the time signal; $u$ is shown in Figure~\ref{fig:numericalRec}(b).
The spatial data $u$ is the right-hand side for the deconvolution problem. 
According to the theory developed in Section~\ref{sec:AnaysisMPIop} the data $u$ is a smoothed version of the particle density $\rho$
in the noiseless case. Here the data $u$ inherits its noise from the noise in the time signal $\hat{s}$.
By solving the Euler-Lagrange equation \eqref{eqn:EulerLagrange} with a CG-iteration we 
obtain the reconstruction of $\rho$ displayed in Figure~\ref{fig:numericalRec}(c).
Finally, Figure~\ref{fig:numericalRec}(d) shows the ground truth, the true particle density $\rho$, of our experiment.

Our result is a very good approximation of the ground truth
and demonstrates the potential of our approach.
Besides smoothing effects due to classical Tychonov regularization (which is a well-known effect and not particular to MPI)
the original density function $\rho$ is well-reconstructed despite the relatively high noise level.

\begin{figure}[t]
	
	\def\figwidth{0.49\columnwidth}
	\def\hs{\hspace{0.0\columnwidth}}
	
	\hs
	\centering
	\begin{subfigure}[t]{\figwidth}
		\centering
		\includegraphics[width=0.9\columnwidth]{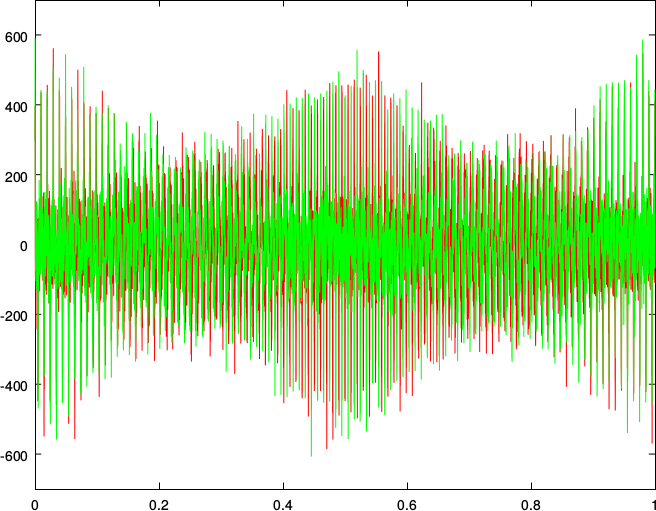}
		\caption{Noisy time signal -- cut-out.}
	\end{subfigure}
	\hs
	\begin{subfigure}[t]{\figwidth}
		\centering
		\includegraphics[width=1.0\columnwidth]{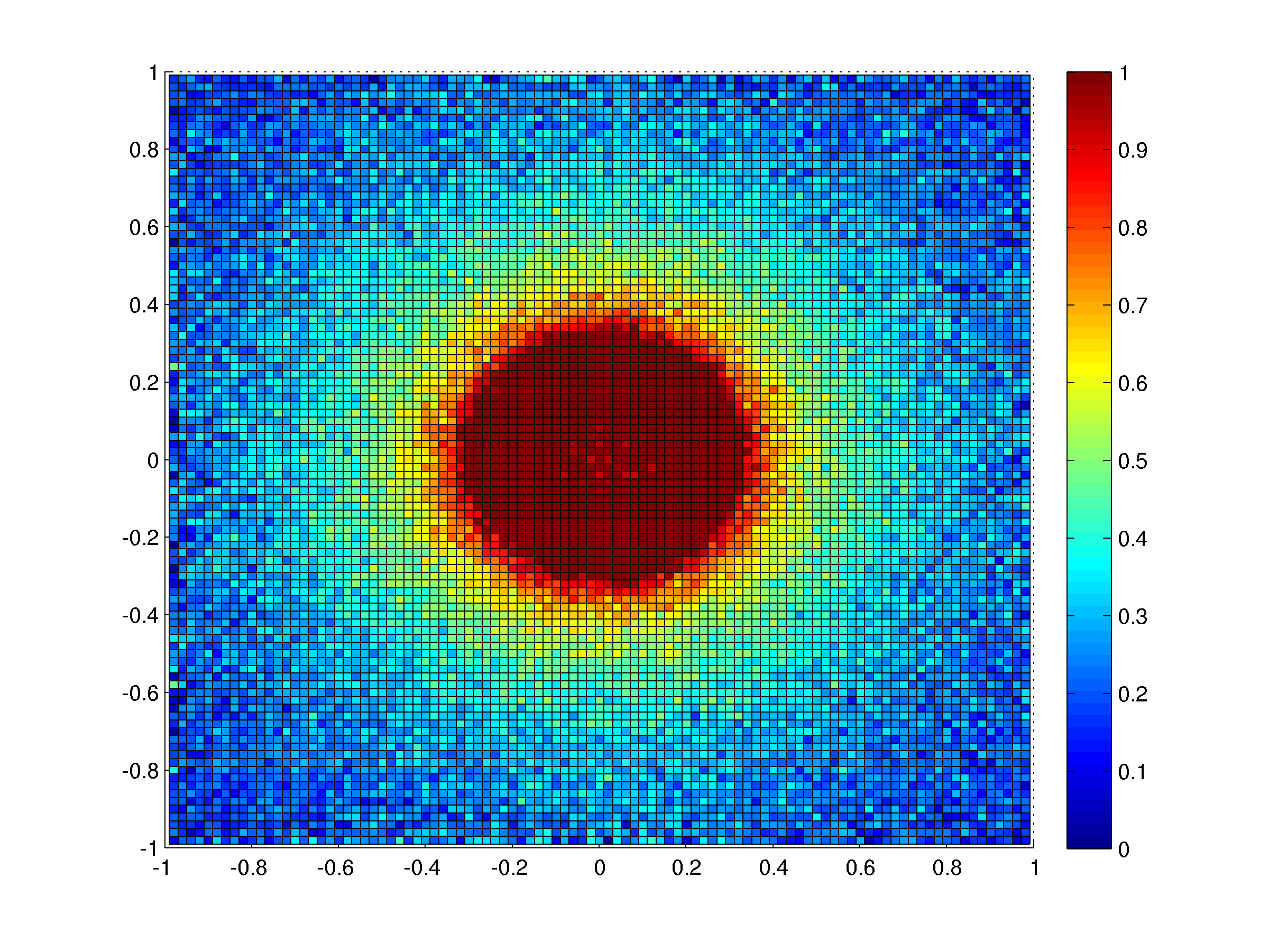}
		\caption{Intermediate step: trace after spatial fitting.}
	\end{subfigure}
	\\[2ex]
	\hs
	\begin{subfigure}[t]{\figwidth}
		\centering
		\includegraphics[width=1.0\columnwidth]{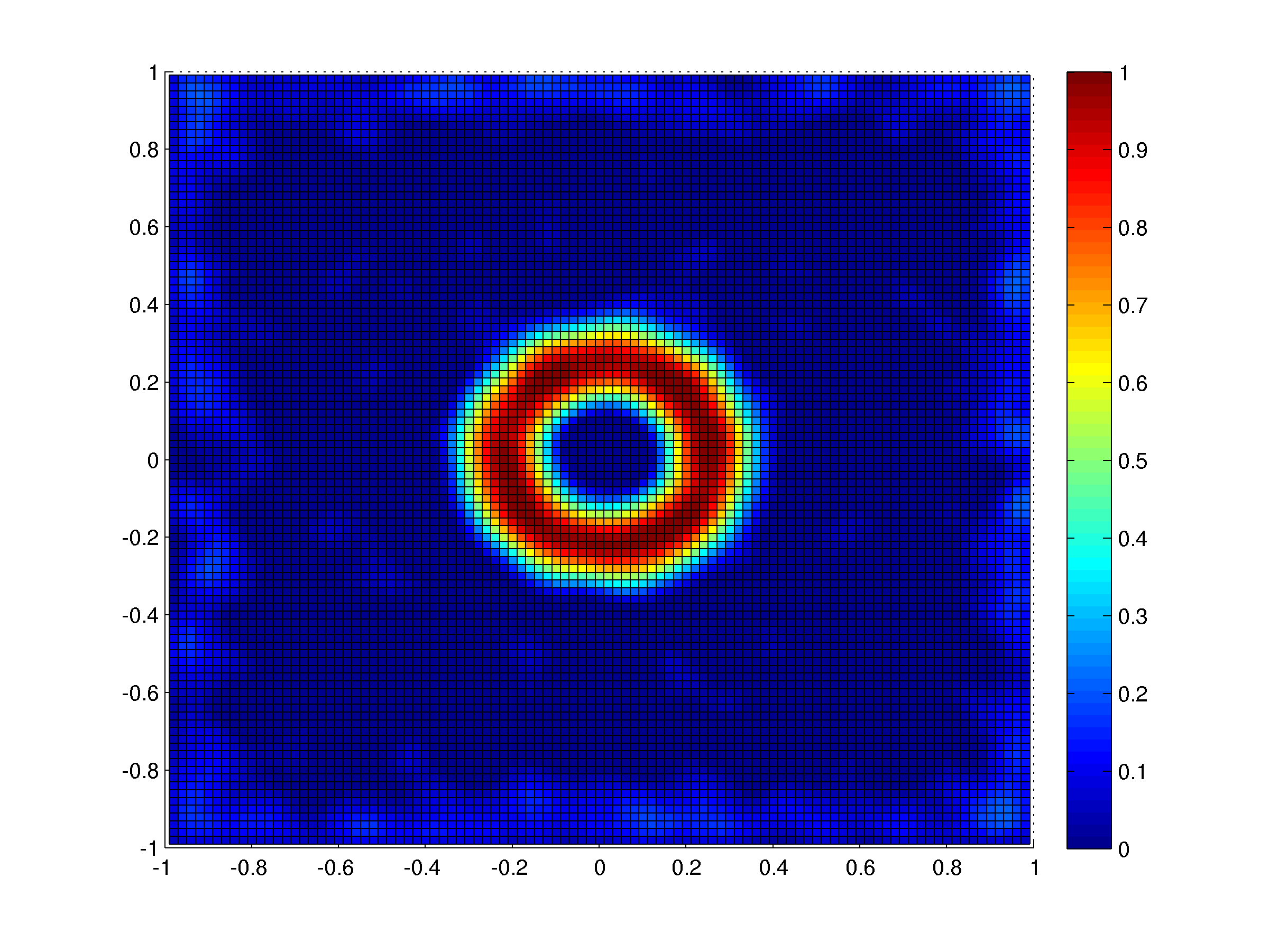}
		\caption{Reconstruction using our method.}
	\end{subfigure}
	\hs
	\begin{subfigure}[t]{\figwidth}
		\centering
		\includegraphics[width=1.0\columnwidth]{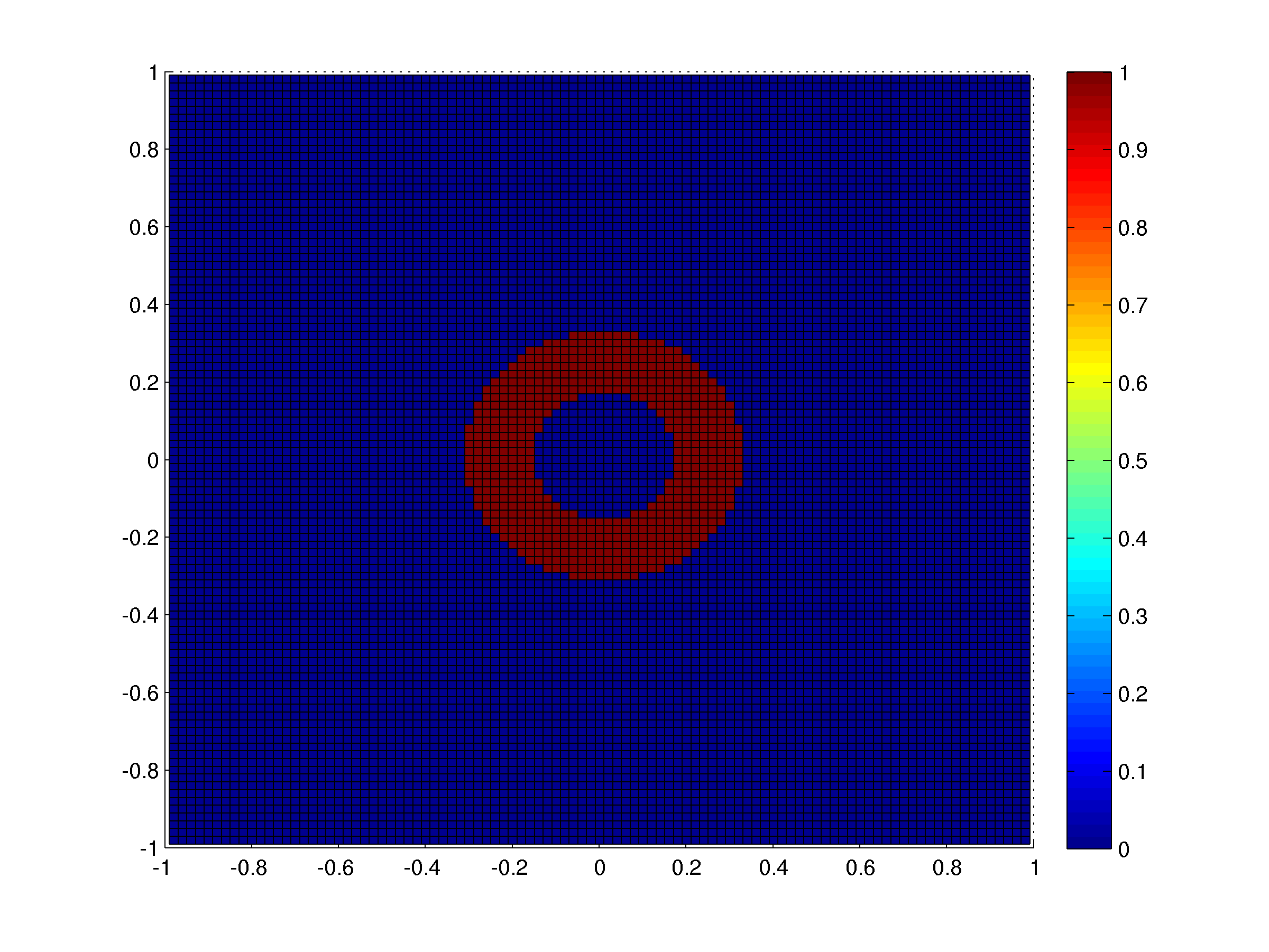}
		\caption{Original.}
	\end{subfigure}

	\caption{{\em Model-based reconstruction for multivariate MPI.} 
		Algorithm~\ref{alg:recAlgo}, proposed in this paper, is used to reconstruct a spatial particle density function $\rho$
		from a noisy MPI time signal.
		(a) shows the nosiy time signal $\hat{s}_k \in \R^2$ (first component in red, second component in green).
		In (b) we see an intermediate step obtained from \eqref{eq:preprocessing2}: 
		this is the data $u$ for our deconvolution problem.
		In a noiseless scenario the data $u$ would be a smoothed version of the true particle density function $\rho$.
		Image (c) shows the reconstructed $\rho$  obtained from deconvolution with the data $u$.
		The underlying ground truth, the true particle density function $\rho$, is displayed in (d). 
		We observe that, besides smoothing effects due to classical Tychonov regularization (which is a well-known effect), 
		the ground truth (d) is well-reconstructed by (c) despite the relatively high noise level.}\label{fig:numericalRec}
\end{figure}

\section*{Acknowledgment}
Thomas M\"arz and  Andreas Weinmann acknowledge funding by the DFG Young Researchers Network ``Mathematics for Magnetic Particle Imaging''.
Andreas Weinmann further acknowledges support by the Helmholtz Association within the young investigator group \mbox{VH-NG-526}.

\bibliographystyle{abbrv}
\bibliography{MPI_arx}{}

\begin{thebibliography}{10}

\bibitem{ambrosio2000functions}
L.~Ambrosio, N.~Fusco, and D.~Pallara.
\newblock {\em Functions of Bounded Variation and Free Discontinuity Problems},
  volume 254.
\newblock Clarendon Press, Oxford, 2000.

\bibitem{bertero1998introduction}
M.~Bertero and P.~Boccacci.
\newblock {\em Introduction to Inverse Problems in Imaging}.
\newblock CRC press, Boca Raton, 1998.

\bibitem{chikazumi1978physics}
S.~Chikazumi and S.~Charap.
\newblock {\em Physics of Magnetism}.
\newblock Krieger Publishing, New York, 1978.

\bibitem{Engl_Regularization_of_Inverse_Problems}
H.~Engl, M.~Hanke, and A.~Neubauer.
\newblock {\em Regularization of {I}nverse {P}roblems}.
\newblock Kluwer Academic Publishers, Dordrecht, 1996.

\bibitem{ferguson2009optimization}
R.~Ferguson, K.~Minard, and K.~Krishnan.
\newblock Optimization of nanoparticle core size for magnetic particle imaging.
\newblock {\em Journal of magnetism and magnetic materials},
  321(10):1548--1551, 2009.

\bibitem{GleichWeizenecker2005}
B.~Gleich and J.~Weizenecker.
\newblock {Tomographic imaging using the nonlinear response of magnetic
  particles}.
\newblock {\em Nature}, 435:1214--1217, 2005.

\bibitem{GleichWeizeneckerBorgert2008}
B.~Gleich, J.~Weizenecker, and J.~Borgert.
\newblock {Experimental results on fast 2D-encoded magnetic particle imaging}.
\newblock {\em Physics in Medicine and Biology}, 53:N81--N84, 2008.

\bibitem{golub2012matrix}
G.~Golub and C.~Van~Loan.
\newblock {\em Matrix Computations}.
\newblock JHU Press, 2012.

\bibitem{GoodwillConolly2010}
P.~Goodwill and S.~Conolly.
\newblock The {X}-space formulation of the magnetic particle imaging process:
  1-{D} signal, resolution, bandwidth, {SNR}, {SAR}, and magnetostimulation.
\newblock {\em IEEE Transactions on Medical Imaging}, 29:1851--1859, 2010.

\bibitem{GoodwillConolly2011}
P.~Goodwill and S.~Conolly.
\newblock Multidimensional {X}-space magnetic particle imaging.
\newblock {\em IEEE Transactions on Medical Imaging}, 30:1581--1590, 2011.

\bibitem{goodwill2012x}
P.~Goodwill, E.~Saritas, L.~Croft, T.~Kim, K.~Krishnan, D.~Schaffer, and
  S.~Conolly.
\newblock X-space mpi: magnetic nanoparticles for safe medical imaging.
\newblock {\em Advanced Materials}, 24(28):3870--3877, 2012.

\bibitem{Gruettner_etal2013}
M.~Gr{\"u}ttner, T.~Knopp, J.~Franke, M.~Heidenreich, J.~Rahmer, A.~Halkola,
  C.~Kaethner, J.~Borgert, and T.~Buzug.
\newblock On the formulation of the image reconstruction problem in magnetic
  particle imaging.
\newblock {\em Biomedical Engineering}, 58(6):583--591, 2013.

\bibitem{jiles1998introduction}
D.~Jiles.
\newblock {\em Introduction to Magnetism and Magnetic Materials}.
\newblock CRC press, 1998.

\bibitem{kirsch2011introduction}
A.~Kirsch.
\newblock {\em An Introduction to the Mathematical Theory of Inverse Problems},
  volume 120.
\newblock Springer, Heidelberg, 2011.

\bibitem{KnoppDiss2010}
T.~Knopp.
\newblock {\em Effiziente Rekonstruktion und alternative Spulentopologien
  f{\"u}r Magnetic Particle Imaging}.
\newblock {Dissertation}, University of L{\"u}beck, 2010.

\bibitem{knopp2011prediction}
T.~Knopp, S.~Biederer, T.~Sattel, M.~Erbe, and T.~Buzug.
\newblock Prediction of the spatial resolution of magnetic particle imaging
  using the modulation transfer function of the imaging process.
\newblock {\em IEEE Transactions on Medical Imaging}, 30(6):1284--1292, 2011.

\bibitem{KnoppBiederer_etal2010}
T.~Knopp, S.~Biederer, T.~Sattel, J.~Rahmer, J.~Weizenecker, B.~Gleich,
  J.~Borgert, and T.~Buzug.
\newblock {2D model-based reconstruction for magnetic particle imaging}.
\newblock {\em Medical Physics}, 37:485--491, 2010.

\bibitem{BuzugKnopp2012}
T.~Knopp and T.~Buzug.
\newblock {\em Magnetic Particle Imaging: An Introduction to Imaging Principles
  and Scanner Instrumentation}.
\newblock Springer, 2012.

\bibitem{Knopp_etal2011hm}
T.~Knopp, M.~Erbe, T.~Sattel, S.~Biederer, and T.~Buzug.
\newblock {A Fourier slice theorem for magnetic particle imaging using a
  field-free line}.
\newblock {\em Inverse Problems}, 27:095004, 2011.

\bibitem{Knopp_etal2010ec}
T.~Knopp, J.~Rahmer, T.~Sattel, S.~Biederer, J.~Weizenecker, B.~Gleich,
  J.~Borgert, and T.~Buzug.
\newblock {Weighted iterative reconstruction for magnetic particle imaging}.
\newblock {\em Physics in Medicine and Biology}, 55:1577--1589, 2010.

\bibitem{KnoppSattel_etal2010}
T.~Knopp, T.~Sattel, S.~Biederer, J.~Rahmer, J.~Weizenecker, B.~Gleich,
  J.~Borgert, and T.~Buzug.
\newblock {Model-Based Reconstruction for magnetic particle imaging}.
\newblock {\em IEEE Transactions on Medical Imaging}, 29:12--18, 2010.

\bibitem{konkle2011development}
J.~Konkle, P.~Goodwill, and S.~Conolly.
\newblock Development of a field free line magnet for projection {MPI}.
\newblock In {\em SPIE Medical Imaging}, page 79650X, 2011.

\bibitem{SPECT}
D.~Kuhl and R.~Edwards.
\newblock Image separation radioisotope scanning.
\newblock {\em Radiology}, 80:653--662, 1963.

\bibitem{Lampe_etal2012}
J.~Lampe, C.~Bassoy, J.~Rahmer, J.~Weizenecker, H.~Voss, B.~Gleich, and
  J.~Borgert.
\newblock {Fast reconstruction in magnetic particle imaging}.
\newblock {\em Physics in Medicine and Biology}, 57:1113--1134, 2012.

\bibitem{lawrence2013catalog}
J.~Lawrence.
\newblock {\em A Catalog of Special Plane Curves}.
\newblock Courier Corporation, 2013.

\bibitem{louis1989inverse}
A.~Louis.
\newblock {\em Inverse und schlecht gestellte Probleme}.
\newblock Teubner, Stuttgart, 1989.

\bibitem{Rahemeretal2009}
J.~Rahmer, J.~Weizenecker, B.~Gleich, and J.~Borgert.
\newblock Signal encoding in magnetic particle imaging: properties of the
  system function.
\newblock {\em BMC Medical Imaging}, 9:4, 2009.

\bibitem{Rahmer_etal2012}
J.~Rahmer, J.~Weizenecker, B.~Gleich, and J.~Borgert.
\newblock {Analysis of a 3-D system function measured for magnetic particle
  imaging}.
\newblock {\em IEEE Transactions on Medical Imaging}, 31(6):1289--1299, 2012.

\bibitem{saritas2013magnetic}
E.~Saritas, P.~Goodwill, L.~Croft, J.~Konkle, K.~Lu, B.~Zheng, and S.~Conolly.
\newblock Magnetic particle imaging ({MPI}) for {NMR} and {MRI} researchers.
\newblock {\em Journal of Magnetic Resonance}, 229:116--126, 2013.

\bibitem{sattel2009single}
T.~Sattel, T.~Knopp, S.~Biederer, B.~Gleich, J.~Weizenecker, J.~Borgert, and
  T.~Buzug.
\newblock Single-sided device for magnetic particle imaging.
\newblock {\em Journal of Physics D: Applied Physics}, 42(2):022001, 2009.

\bibitem{Schomberg2010}
H.~Schomberg.
\newblock {Magnetic particle imaging: model and reconstruction}.
\newblock In {\em 2010 IEEE International Symposium on Biomedical Imaging: From
  Nano to Macro}, pages 992--995, 2010.

\bibitem{PET}
M.~Ter-Pogossian, M.~Phelps, E.~Hoffman, and N.~Mullani.
\newblock A positron-emission transaxial tomograph for nuclear imaging
  {(PETT)}.
\newblock {\em Radiology}, 114:89--98, 1975.

\bibitem{WeizeneckerBorgertGleich2007}
J.~Weizenecker, J.~Borgert, and B.~Gleich.
\newblock {A simulation study on the resolution and sensitivity of magnetic
  particle imaging}.
\newblock {\em Physics in Medicine and Biology}, 52:6363--6374, 2007.

\bibitem{Weizenecker_etal2009}
J.~Weizenecker, B.~Gleich, J.~Rahmer, H.~Dahnke, and J.~Borgert.
\newblock {Three-dimensional real-time in vivo magnetic particle imaging}.
\newblock {\em Physics in Medicine and Biology}, 54:L1--L10, 2009.

\bibitem{zheng2013quantitative}
B.~Zheng, T.~Vazin, W.~Yang, P.~Goodwill, E.~Saritas, L.~Croft, D.~Schaffer,
  and S.~Conolly.
\newblock Quantitative stem cell imaging with magnetic particle imaging.
\newblock In {\em IEEE International Workshop on Magnetic Particle Imaging},
  page~1, 2013.

\end{thebibliography}

\end{document}